\documentclass{amsart}
\usepackage{latexsym}
\usepackage{amssymb, amsbsy, amsthm, amsmath, amstext, amsopn, verbatim}
\usepackage[color,all]{xy}
\usepackage{amsfonts}
\usepackage{amscd}
\usepackage{mathrsfs}
\usepackage{verbatim}
\usepackage{xcolor}
\usepackage{tabularx}
\usepackage{mathtools}
\usepackage{hyperref}
\usepackage{enumerate}

\input xy
\xyoption{all}

\usepackage{parcolumns}

\newcommand{\bd}{\begin{displaymath}}
\newcommand{\ed}{\end{displaymath}}
\newcommand{\bcd}{\begin{CD}}
\newcommand{\ecd}{\end{CD}}

\newcommand{\on}{\operatorname}

\title[Mirror symmetry for parabolic Higgs bundles via $p$-adic integration]{Mirror symmetry for parabolic Higgs bundles via $p$-adic integration}
\author{Shiyu Shen}
\address{The Institute of Science and Technology Austria\\Klosterneuburg 3400, Austria}
\email{shiyu.shen@ist.ac.at}

\thanks{
Shiyu Shen has received funding from the European Union's Horizon 2020 research and innovation program under the Marie Skłodowska-Curie grant agreement No. 101034413.
}

\numberwithin{equation}{section}
\newtheorem{theorem}{Theorem}[section]
\newtheorem{example}[theorem]{Example}
\newtheorem{lemma}[theorem]{Lemma}

\newtheorem{prop}[theorem]{Proposition}
\newtheorem{corollary}[theorem]{Corollary}

\theoremstyle{definition}
\newtheorem{definition}[theorem]{Definition}

\theoremstyle{remark}
\newtheorem{remark}[theorem]{Remark}

\begin{document}
\maketitle

\begin{abstract}
Applying the technique of $p$-adic integration, we prove the topological mirror symmetry conjecture of Hausel-Thaddeus for the moduli spaces of (strongly) parabolic Higgs bundles for the structure groups $\on{SL}_n$ and $\on{PGL}_n$, building on previous work of Groechenig-Wyss-Ziegler on the non-parabolic case. We also prove the $E$-polynomial of the smooth moduli space of parabolic $\on{GL}_n$-Higgs bundles is independent of the degree of the underlying vector bundles.
\end{abstract}

\section{Introduction}
\subsection{Topological mirror symmetry for Higgs bundles}
The notion of a Higgs bundle was introduced by Hitchin in his seminal paper \cite{self dual} with the motivation of studying certain differential equations from gauge theory. These geometric objects have an extremely rich structure and draw intense research interest from different areas of mathematics. One of the most intriguing feature of the moduli space $\mathcal{M}$ of Higgs bundles over a smooth projective curve $X$ is that it has the structure of a completely integrable system. This structure is induced by the so-called Hitchin map from $\mathcal{M}$ to an affine space $\mathcal{A}$: coordinate functions on $\mathcal{A}$ induce Poisson-commuting functions on $\mathcal{M}$, and a general fiber of this map is an abelian variety, i.e. a compact torus which is also a complex algebraic variety. 

Let $G$ be a complex reductive group and $\breve{G}$ its Langlands dual group. The moduli space $\mathcal{M}_G$ of $G$-Higgs bundles and the moduli space $\mathcal{M}_{\breve{G}}$ of $\breve{G}$-Higgs bundles are mapped to the same affine space $\mathcal{A}$ via the Hitchin map: 
\bd
\xymatrix{
\mathcal{M}_G\ar[dr]^{h_G}&& \mathcal{M}_{\breve{G}}\ar[dl]_{h_{\breve{G}}}\\
&\mathcal{A}.
}
\ed
These two integrable systems are dual to each other, in the sense that for a general $a\in \mathcal{A}$, the fibers $h_G^{-1}(a)$ and $h_{\breve{G}}^{-1}(a)$ are dual abelian varieties\footnote{Strictly speaking, these Hitchin fibers are gerbes over abelian varieties, and this duality should be interpreted as a duality between commutative group stacks.}. This statement is first proved by Hausel and Thaddeus \cite{HT} for $G=\on{SL}_n$ and $\breve{G}=\on{PGL}_n$, later generalized to all $G$ by
Donagi and Pantev \cite{DP}. Motivated by this duality between Hitchin integrable systems and the Strominger-Yau-Zaslow picture of mirror symmetry, Hausel and Thaddeus proposed in \cite{HT} the so-called topological mirror symmetry conjecture, in which they predict a correspondence between the (appropriately defined) Hodge numbers of those moduli spaces $\mathcal{M}_{\on{SL}_n}^d$ and $\mathcal{M}_{\on{PGL}_n}^e$ (here $d$ and $e$ stands for the degree of the underlying vector bundles, and we assume they are coprime to $n$). This conjecture (in the non-parabolic setting) has been proved by Groechenig, Wyss and Ziegler \cite{MWZ} through the application of $p$-adic integration to the Hitchin integrable systems. Developing on the beautiful connections between $p$-adic integration, point-counting and Betti numbers established by Weil, they reduce the job of matching the Hodge numbers of $\mathcal{M}_{\on{SL}_n}^d$ and $\mathcal{M}_{\on{PGL}_n}^e$ to matching the integral of certain measures on the $p$-adic version of the moduli spaces. The structure of the Hitchin fibrations allows this comparison to be further reduced to a number theoretic computation concerning dual abelian varieties over local fields. See Section \ref{section: p-adic integration on Hitchin systems} for the structure of this argument. In \cite{MS}, the authors provided a separate proof of this conjecture using sheaf-theoretic methods. In \cite{GO}, the authors established the topological mirror symmetry for parabolic Higgs bundles of rank $n=2,3$ with full flags. 

\subsection{Mirror symmetry for parabolic Higgs bundles}
The goal of this paper is to prove the topological mirror symmetry conjecture for the moduli spaces of parabolic Higgs bundles. Let $X$ be a smooth projective curve over the field $\mathbb{C}$ of complex numbers. Let $D=q_1+q_2+\cdots+q_m$ be an effective reduced divisor on $X$, and let $P_D = (P_1, P_2, \dots, P_m)$ be an ordered $m$-tuple of parabolic subgroups of $\on{GL}_n(\mathbb{C})$. We fix a line bundle $M$ on $X$ that either satisfies $\on{deg}M>2g-2$ or equals the canonical bundle $K$. We consider the moduli space $\mathcal{M}_{\on{GL}_n}^{d, \boldsymbol{\alpha}}$ of semistable $M$-twisted parabolic $\on{GL}_n$-Higgs bundles of degree $d$ on $X$, the moduli space $\mathcal{M}_{\on{SL}_n}^{d,\boldsymbol{\alpha}}$ of semistable $M$-twisted parabolic ``$\on{SL}_n$''-Higgs bundles of degree $d$ and the moduli stack $\mathcal{M}_{\on{PGL}_n}^{d,\boldsymbol{\alpha}}$ of semistable $M$-twisted parabolic $\on{PGL}_n$-Higgs bundles of ``degree'' $d$. The precise definition of those moduli spaces are given in Section \ref{section: moduli spaces of parabolic Higgs bundles}. Here $\boldsymbol{\alpha}$ is a set of real numbers called parabolic weights which plays a role in defining the notion of stability on those objects. We say $\boldsymbol{\alpha}$ is generic for degree $d$ if a semistable parabolic vector bundle of degree $d$ is automatically stable. 

Let $\boldsymbol\alpha$ be a set of parabolic weights generic for degree $d$. With this assumption, both $\mathcal{M}_{\on{GL}_n}^{d, \boldsymbol{\alpha}}$ and $\mathcal{M}_{\on{SL}_n}^{d,\boldsymbol{\alpha}}$ are smooth quasi-projective varieties, and $\mathcal{M}_{\on{PGL}_n}^{d,\boldsymbol{\alpha}}$ is a smooth Deligne-Mumford stack. All three moduli spaces admit the Hitchin map to an affine space $\mathcal{A}_P$. The generic fiber of the parabolic Hitchin map is described in \cite{shen}: there exists an open subset $\mathcal{A}_P^0\subset \mathcal{A}_P$ of the Hitchin base such that the Hitchin fiber for the $\on{GL}_n$-moduli space $\mathcal{M}_{\on{GL}_n}^{\boldsymbol\alpha}$ is isomorphic to the relative Picard scheme $\on{Pic}(\Sigma/\mathcal{A})$, where $\Sigma/\mathcal{A}$ is a flat family of smooth projective curves constructed from the family of spectral curves by consecutive blow-ups, see Theorem \ref{theorem: spectral data GLn}. For the $\on{SL}_n$-moduli space $\mathcal{M}_{\on{SL}_n}^{d,\boldsymbol{\alpha}}$ and the $\on{PGL}_n$-moduli stack $\mathcal{M}_{\on{PGL}_n}^{d,\boldsymbol{\alpha}}$, the generic Hitchin fibers over $\mathcal{A}_P^0$ are torsors for dual abelian schemes. 

Let $\boldsymbol\alpha$ (resp. $\boldsymbol\alpha'$) be a set of parabolic weights generic for degree $d$ (resp. $e$), we prove the following\footnote{Our proof relies on the assumption that a certain open subset of the parabolic Hitchin base defined in \cite{shen} is non-empty. See Remark \ref{rk: singular spectral curves}. This is always satisfied when $g_X\geq 2$. When the genus is $1$ or $0$, we need to put restrictions on the number of marked points and the type of parabolic subgroups involved, see \cite{shen} Section 1.2.}

\begin{theorem}[cf. Theorem \ref{thm: main theorem} and Theorem \ref{thm: main theorem independence} (a)]\label{thm intro: parabolic mirror symmetry}
    
    (a) We have an equality  $E(\mathcal{M}_{\on{SL}_n}^{d, \boldsymbol\alpha}; u,v)=E_{st}(\mathcal{M}_{\on{PGL}_n}^{e, \boldsymbol\alpha'},\alpha_d; u,v)$ of $E$-polynomials. The right-hand side of the equality stands for the twisted stringy $E$-polynomial with respect to a gerbe $\alpha_d$ on $\mathcal{M}^{e, \boldsymbol\alpha'}_{\on{PGL}_n}$.
    
    (b) The $E$-polynomial $E(\mathcal{M}_{\on{GL}_n}^{d, \boldsymbol\alpha}; u,v)$ is independent of the degree $d$ and the parabolic weights $\boldsymbol\alpha$. 
\end{theorem}

When $X$ is a smooth projective curve over a finite field $k$, assuming the set of parabolic weights $\boldsymbol\alpha$ is generic for $d$, we also have the following result concerning point-counts

\begin{theorem}[cf. Theorem \ref{thm: main theorem independence} (b)] \label{thm intro: independence in char p}
Assume $char(k)>n$. Then the point-count $\#\mathcal{M}_{\on{GL}_n}^{d, \boldsymbol\alpha}(k)$ is
independent of d and $\boldsymbol\alpha$.
\end{theorem}

Our proof of Theorem \ref{thm intro: parabolic mirror symmetry} and Theorem \ref{thm intro: independence in char p} is based on the same strategy as in \cite{MWZ}, but some new ingredients come into play. The arguments in \cite{MWZ} rely heavily on the existence of the so-called regular locus $\mathcal{M}^{reg}\subset \mathcal{M}$\footnote{Here $\mathcal{M}$ stands for the moduli space of $\on{GL}_n$-, $\on{SL}_n$- or $\on{PGL}_n$-Higgs bundles in the non-parabolic setting.} of the moduli space, which is characterized by the Higgs field being regular everywhere on the curve $X$. This open subset $\mathcal{M}^{reg}$ forms a torsor for a certain commutative group scheme over the entire Hitchin base, and the complement satisfies $\on{codim}(\mathcal{M}\backslash\mathcal{M}^{reg})\geq 2$. This locus $\mathcal{M}^{reg}$ serves multiple purposes in \cite{MWZ}: (a) it is used to construct the gauge form $\omega$ on $\mathcal{M}$ that they do $p$-adic integration with. This gauge form $\omega$ automatically satisfies the property that the $p$-adic volume of a general Hitchin fiber is independent of the degree of the underlying vector bundle (as long as the Hitchin fiber is non-empty). (b) when working over a local field $F$ with ring of integers $\mathcal{O}_F$, the regular locus is used to show that for both the $\on{GL}_n$- and the $\on{PGL}_n$-moduli space, the Hitchin fiber admits an $\mathcal{O}_F$-rational point over any $\mathcal{O}_F$-point of the Hitchin base, and this $\mathcal{O}_F$-rational point can be chosen such that the underlying vector bundle is of any degree. This statement plays an important role in the proof of both the topological mirror symmetry and the independence on degree for $E$-polynomials and point-counts in the non-parabolic setting. 

Unfortunately, we don't have this regular locus in the parabolic setting unless all parabolic subgroups involved are Borel, and even if all parabolic subgroups are Borel, we don't have $\on{codim}(\mathcal{M}\backslash\mathcal{M}^{reg})\geq 2$. 
In Section \ref{Section: Volume forms}, we construct a gauge form $\omega$ on $\mathcal{M}$\footnote{Here $\mathcal{M}$ stands for $\mathcal{M}_{\on{GL}_n}^{d, \boldsymbol\alpha}$, $\mathcal{M}_{\on{SL}_n}^{d, \boldsymbol\alpha}$ or $\mathcal{M}_{\on{PGL}_n}^{d, \boldsymbol\alpha}$.} using a different strategy: we relate the moduli space $\mathcal{M}$ with the cotangent bundle $T^{*}\mathcal{N}$ of the moduli space $\mathcal{N}$ of vector bundles with certain types of parabolic and level structures. We obtain the desired gauge form by manipulating the symplectic form on $T^{*}\mathcal{N}$. 

The relation between stringy $E$-polynomials, stringy point-counts and $p$-adic integration reviewed in Section \ref{section: $E$-polynomials, point-counting and $p$-adic integration} reduces the comparison of stringy $E$-polynomials and point-counts in Theorem \ref{thm intro: parabolic mirror symmetry} and Theorem \ref{thm intro: independence in char p} to the comparison of the $p$-adic integral of $\omega$ on the involved moduli spaces. By throwing away a subset of zero measure, this is further reduced to comparing the $p$-adic volume of the Hitchin fibers over $F$-points of $\mathcal{A}^0_P\subset \mathcal{A}_P$. For Theorem \ref{thm intro: parabolic mirror symmetry} (b) and Theorem \ref{thm intro: independence in char p}, the main result of \cite{BB} plays a crucial role to guarantee that the gauge form $\omega$ behaves well when we change the degree $d$ of the underlying vector bundles, see Subsection \ref{subsection: comparison between degrees}. For Theorem \ref{thm intro: parabolic mirror symmetry} (a), we establish a stacky version of the torsor-gerbe duality in \cite{HT} (Proposition 3.2 and Proposition 3.6), see Theorem \ref{thm: gerbe duality stronger version}. This Theorem \ref{thm: gerbe duality stronger version} together with computations in \cite{MWZ} guarantees the equality of $p$-adic integrals on the $\on{SL}_n$- and $\on{PGL}_n$-side.

\subsection{Structure of the article}
In Section \ref{section: $E$-polynomials, point-counting and $p$-adic integration}, we review the relation between stringy $E$-polynomials, stringy point-counts and $p$-adic integration over certain smooth abelian Deligne-Mumford stacks. In Section \ref{section: p-adic integration on Hitchin systems}, we describe the framework for the application of $p$-adic integration to Hitchin systems. In Section \ref{section: moduli spaces of parabolic Higgs bundles}, we first define the $\on{GL}_n$-, $\on{SL}_n$ and $\on{PGL}_n$- moduli spaces involved in our main theorems. Then we describe the generic Hitchin fiber in all three cases. Then we establish the stacky version of the torsor-gerbe duality in \cite{HT}, see Theorem \ref{thm: gerbe duality stronger version}. In Section \ref{Section: Volume forms}, we first construct a gauge form on the moduli space of parabolic Higgs bundles. Then we show this gauge form behaves well when we change the degree $d$ of the underlying vector bundles. Both Theorem \ref{thm intro: parabolic mirror symmetry} and Theorem \ref{thm intro: independence in char p} are proved in Section \ref{section: proof of mirror symmetry}.  

\subsection{Acknowledgements}
I would like to thank Michael Groechenig for many helpful discussions on this subject, and for his contribution to the proof of Proposition \ref{prop: trializing section exists}. I would like to thank Tamas Hausel, Anton Mellit and Andr\'e Oliveira for helpful conversations on related subjects.

\section{$E$-polynomials, point-counting and $p$-adic integration}\label{section: $E$-polynomials, point-counting and $p$-adic integration}
\subsection{$E$-polynomials and point-counting}
The goal of this subsection is to review the relation between stringy $E$-polynomials and stringy point-counts described in \cite{MWZ} Section 2. We start by recalling the definition of $E$-polynomial of complex quasi-projective varieties. Let $X$ be a complex smooth projective variety. The $E$-polynomial of $X$ is defined to be 
\[
E(X; u,v)=\sum(-1)^{p+q}h^{p,q}(X)u^pv^q,
\]
where $h^{p,q}(X)$ is the $(p,q)$-th Hodge number of $X$. 
This definition extends to all complex quasi-projective varieties by imposing the condition that for any closed subvariety $Z\subset X$, we have
\[
 E(X;u,v)=E(X\backslash Z; u,v)+E(Z; u,v). 
\]
The precise formula of $E(X; u,v)$ in terms of mixed Hodge numbers can be found in \cite{HRV} Definition 2.1.4. For a quotient stack of the form $\mathcal{X}=[Y/\Gamma]$, where $Y$ is a complex quasi-projective variety with the action of a finite group $\Gamma$, the $E$-polynomial of $\mathcal{X}$ is defined to be the $\Gamma$-invariant part of $E(Y;u,v)$. 

Now we assume $\mathcal{X}$ is smooth. Motivated by stringy Hodge numbers considered in \cite{BD}, Hausel and Thaddeus \cite{HT} defined the \emph{stringy $E$-polynomial} of $\mathcal{X}$ to be
\begin{equation}\label{eq: stringy E-polynomial}
   E_{st}(\mathcal{X}; u,v)=\sum_{\gamma\in \Gamma_{\on{conj}}} (\sum_{\mathcal{Z}\in \pi_0([Y^{\gamma}/C(\gamma)])} E(\mathcal{Z};u,v)(uv)^{F(\gamma, \mathcal{Z})}), 
\end{equation}
where the first summation is over the set $\Gamma_{\on{conj}}$ of conjugacy classes in $\Gamma$, the second summation is over connected components of the quotient stack $[Y^{\gamma}/C(\gamma)]$, and $F(\gamma, \mathcal{Z})$ is the so-called fermionic shift, see \cite{MWZ} Definition 2.2. If we shift our base field from $\mathbb{C}$ to a finite field $\mathbb{F}_q$ of order $q$, similar formula as in $(\ref{eq: stringy E-polynomial})$ is used to define the \emph{stringy point-count} of $\mathcal{X}=[Y/\Gamma]$ as
\begin{equation}\label{eq: stringy point count}
   \#_{st}(\mathcal{X})=\sum_{\gamma\in \Gamma_{\on{conj}}} (\sum_{\mathcal{Z}\in \pi_0([Y^{\gamma}/C(\gamma)])} \# \mathcal{Z}(\mathbb{F}_q)\cdot q^{F(\gamma, \mathcal{Z})}).  
\end{equation}
Here  the point-count $\# \mathcal{Z}(\mathbb{F}_q)$ is defined by 
\[
\# \mathcal{Z}(\mathbb{F}_q)=\sum_{z\in \mathcal{Z}(\mathbb{F}_q)_{\on{iso}}}\frac{1}{\on{Aut}_{ \mathcal{Z}(\mathbb{F}_q)}(z)},
\]
where $\mathcal{Z}(\mathbb{F}_q)_{\on{iso}}$ stands for the set of isomorphic classes in $\mathcal{Z}(\mathbb{F}_q)$.

With an eye towards the formulation of topological mirror symmetry, we further consider a smooth quotient stack $\mathcal{X}=[Y/\Gamma]$ together with a $\mu_r$-gerbe $\alpha$ over $\mathcal{X}$. By the discussion in \cite{HT} Section 4, this $\mu_r$-gerbe $\alpha$ leads to a $\mu_r$-torsor over the inertia stack $I\mathcal{X}=\displaystyle\coprod_{\gamma\in \Gamma_{\on{conj}}}[Y^\gamma/C(\gamma)]$, which we denote by $\mathcal{L}$. This $\mu_r$-torsor $\mathcal{L}$ can be used to define a twisted version of the $E$-polynomial in (\ref{eq: stringy E-polynomial}) and the point-count in (\ref{eq: stringy point count}), which leads to the so-called \emph{twisted stringy $E$-polynomial} $E_{st}(\mathcal{X},\alpha; u,v)$ for $\mathcal{X}$ over $\mathbb{C}$ and \emph{twisted stringy point-count} $\#^{\alpha}_{st}(\mathcal{X})$ for $\mathcal{X}$ over $\mathbb{F}_q$. We refer the readers to \cite{MWZ} Definition 2.12 and 2.13 for the precise definition of $E_{st}(\mathcal{X},\alpha; u,v)$ and $\#^{\alpha}_{st}(\mathcal{X})$. 

Now we record the relation between twisted stringy $E$-polynomial and twisted stringy point-count. 
\begin{theorem}[cf. \cite{MWZ} Theorem 2.19]\label{thm: stringy point count vs E-polynomial}
Let $R\subset \mathbb{C}$ be a subalgebra of finite type over $\mathbb{Z}$. Let $Y_1$ and $Y_2$ be two smooth $R$-varieties acted on by two finite abelian groups $\Gamma_1$ and $\Gamma_2$ respectively. For $i=1,2$, let $\mathcal{X}_i=[Y_i/\Gamma_i]$ be the corresponding quotient stack, and let $\alpha_i$ be a $\mu_r$-gerbe on $\mathcal{X}_i$. If for any ring homomorphism $R\longrightarrow \mathbb{F}_q$ from $R$ to a finite field $\mathbb{F}_q$ we have equality $\#_{st}^{\alpha_1}(\mathcal{X}_1\times_{R}\mathbb{F}_q)=\#_{st}^{\alpha_2}(\mathcal{X}_2\times_{R}\mathbb{F}_q)$ of twisted stringy point-counts, then we also have the following equality of twisted stringy $E$-polynomials:
\[
    E_{st}(\mathcal{X}_1\times_{R}\mathbb{C}, \alpha_1; u,v)=E_{st}(\mathcal{X}_2\times_{R}\mathbb{C}, \alpha_2; u,v).
\]
\end{theorem}

\subsection{Point-counting and $p$-adic integration }\label{subsection: p-adic integration and point counting}
In this subsection we review the theory of $p$-adic integration over certain Deligne-Mumford stacks and its relation to stringy point-counts. The main references are \cite{MWZ}, \cite{GWZ} and \cite{yas}. We fix $F$ to be a non-Archimedean local field with ring of integers $\mathcal{O}_F$ and residue field $k_F$ of characteristic $p$ and order $q$. For any positive integer $n$, $F^n$ equipped with the non-Archimedean absolute value becomes a locally compact topological group, therefore admits a unique Haar measure $\mu$ that satisfies $\mu(\mathcal{O}_F^n)=1$. Now let $X$ be a smooth variety over $F$ with a volume form $\omega \in \Gamma(X, \Omega^{\on{top}}_{X/F})$. We say $\omega$ is a \emph{gauge form} if $\omega$ is a trivialising section of the invertible sheaf $\Omega^{\on{top}}_{X/F}$. For each gauge form $\omega$, we associate with it a measure $\mu_{\omega}$ on the set $X(F)$ of $F$-points on $X$. The measure $\mu_{\omega}$ is defined locally: for any open chart $U\hookrightarrow F^n$ of $X(F)$, assuming $\omega=fdx_1\wedge dx_2\wedge\cdots\wedge dx_n$ on this open chart,  we define
$$\mu_{\omega}(U)=\int_{U}|\omega|=\int_{U}|f|d\mu,$$ 
where the right-hand side stands for the integral of the absolute value of $f$ with respect to the Haar measure $\mu$. 

Now let $X$ be a smooth variety over $\mathcal{O}_F$. In this case, the $p$-adic measure $\mu_{\omega}$ does not depend on the choice of the gauge form $\omega \in \Gamma(X, \Omega^{\on{top}}_{X/\mathcal{O}_F})$ when restricted to the set $X(\mathcal{O}_F)$ of $\mathcal{O}_F$-points on $X$. Therefore $p$-adic measures on local charts of $X$ glue together to form a measure on $X(\mathcal{O}_F)$ which we call $\mu_{can}$. The following theorem of Weil builds the bridge between point-counting and $p$-adic integration. 
\begin{theorem}[cf. \cite{Weil} Theorem 2.2.5]
Let $X$ be a smooth variety over $\mathcal{O}_F$ of dimension $n$. Then we have $\displaystyle\frac{\#X(k_F)}{q^n}=\int_{X(\mathcal{O}_F)}d\mu_{can}$.
\end{theorem}

We record the follow properties of the $p$-adic measure $\mu_{\omega}$ that play a significant role in simplifying the computation of $p$-adic integration for Hitchin systems.

\begin{prop}[\cite{yas} Lemma 4.3 and \cite{MWZ} Proposition 4.1]\label{prop: properties of p-adic integration}
(a) Let $X$ be a smooth $\mathcal{O}_F$-variety with a gauge form $\omega$ on $X_F\coloneqq X\times_{\mathcal{O}_F}{F}$, and let $Y\subset X$ be a closed subscheme of positive codimension. Then $\mu_{\omega}(Y(\mathcal{O}_F))=0.$

(b) Let $h: X\to Y$ be a smooth morphism between smooth $F$-varieties. Let $\omega_X$ (resp. $\omega_Y$) be a gauge form on $X$ (resp. $Y$). Let $\theta\in \Gamma(X, \Omega_{X/Y}^{\on{top}})$ be the unique relative top-degree form that satisfies $\omega_X=h^*\omega_Y\wedge \theta$. For any integrable function $f$ on $X(F)$, we have the following equality
\[
\int_{X(F)} f d\mu_{\omega_X}=\int_{Y(F)\ni y}(\int_{h^{-1}(y)(F)}fd\mu_{\theta_y})d\mu_{\omega_Y},
\]
where $\theta_y\in\Gamma(h^{-1}(y), \Omega_{h^{-1}(y)/F}^{\on{top}})$ stands for the pull-back of $\theta$ to $h^{-1}(y)$.
\end{prop}

The theory of $p$-adic integration for smooth $F$-varieties is generalized to certain Deligne-Mumford stacks in \cite{yas}. Following \cite{MWZ}, we restrict ourselves to the following situation:
\begin{definition}\label{def: admissible stack}
        A finite abelian quotient stack $\mathcal{M}$ over $\mathcal{O}_F$ is called \emph{admissible} if it admits a presentation $\mathcal{M}\cong [Y/\Gamma]$, where $Y$ is smooth quasi-projective $\mathcal{O}_F$-variety with a generically free action by a finite abelian group $\Gamma$, such that the order $|
\Gamma|$ of $\Gamma$ is coprime to $p$ and $\mathcal{O}_F$ contains all the $|\Gamma|$-th roots of unity. 
\end{definition}
Let $U\subset Y$ be the locus where the action of $\Gamma$ is free. Let $M$ be the geometric quotient of $Y$
by $\Gamma$ and let $\text{pr}: Y\longrightarrow M$ be the quotient map. In this setting, one can define a measure $\mu_{orb}$ on $M(\mathcal{O}_F)^{\sharp}\coloneqq M(\mathcal{O}_F)\cap \text{pr}(U)(F)$ which is called the orbifold measure. We refer the readers to \cite{MWZ} and \cite{yas} for the precise definition of this measure $\mu_{orb}$. For our applications to Higgs bundles, we will only need the following simple description of $\mu_{orb}$ in a special setting. 

\begin{lemma}[cf. \cite{MWZ} Remark 4.13]\label{lemma: orbifold measure special description}
We assume $\Omega^{\on{top}}_Y$ admits a $\Gamma$-invariant trivializing section $\omega$. Then $\omega$ descends to a section $\omega_{orb}$ of $\Omega^{\on{top}}_{\on{pr}(U)}$, and the orbifold measure $\mu_{orb}$ is given by integrating $|\omega_{orb}|$ on $M(\mathcal{O}_F)^{\sharp}\subset \on{pr}(U)(F)$. 
\end{lemma}

Now let $\alpha$ be a $\mu_r$-gerbe on an admissible finite abelian quotient stack $\mathcal{M}$. We further assume that $\mathcal{O}_F$ contains all the $r$-th roots of unity and $r$ is coprime to the characteristic of the residue field. Following \cite{MWZ}, we define a function $f_\alpha: M(\mathcal{O}_F)^{\sharp}\longrightarrow \mathbb{C}$ as follows. Let $x\in M(\mathcal{O}_F)^{\sharp}$ with corresponding $x_F\in \on{pr}(U)(F)$. We have the following pull-back diagram 
\bd
\xymatrix{
\on{Spec}(F)\times_{\on{pr}(U)}U \ar[r] \ar[d] & U\ar[d]^{\on{pr}}  \\
\on{Spec}(F) \ar[r]^{x_F} & \on{pr}(U).
}
\ed
Since the $\Gamma$-action on $U$ is free, the vertical arrow on the left-hand side gives a $\Gamma$-torsor over $\on{Spec}(F)$. By the definition of the quotient stack $\mathcal{M}=[Y/\Gamma]$, this pull-back diagram gives a lifting of $x_F\in \on{pr}(U)(F)$ to an $F$-point $\tilde{x}_F\in \mathcal{M}(F)$ of $\mathcal{M}$. Pulling back along $\tilde{x}_F$, we get a $\mu_r$-gerbe $\tilde{x}_F^{*}\alpha$ over $\on{Spec}(F)$, which gives an element in the Brauer group $\on{Br}(F)$. The Brauer group of a non-Archimedean local field $F$ is isomorphic to $\mathbb{Q}/\mathbb{Z}$ via the Hasse invariant $\text{inv}:Br(F)\longrightarrow \mathbb{Q}/\mathbb{Z}$, see \cite{serre}. Now we define $f_\alpha: M(\mathcal{O}_F)^{\sharp}\longrightarrow \mathbb{C}$ by the following formula
\begin{equation}\label{eq: f_gerbe}
f_\alpha(x)=\on{exp}(2\pi i \cdot\text{inv}(\tilde{x}_F^{*}\alpha)).
\end{equation}
The relation between stringy point-count and $p$-adic integration is described in the following theorem. 

\begin{theorem}[cf. \cite{MWZ} Corollary 5.28]\label{thm: stringy point count=p-adic integral}
Let $\mathcal{M}$ be an admissible finite abelian quotient stack over $\mathcal{O}_F$ with a $\mu_r$-gerbe $\alpha$ and the associated function $f_\alpha$ as described above. Then the stringy point-count of the special fiber $\mathcal{M}_{k_F}\coloneqq \mathcal{M}\times_{\mathcal{O}_F}k_F$ can be computed using the following equation
\[
   \frac{ \#^{\alpha}_{st}(\mathcal{M}_{k_F})}{q^{\on{dim}\mathcal{M}}}=\int_{{M(\mathcal{O}_F})^{\sharp}}\bar{f_\alpha}d\mu_{orb},
\]
where $\bar{f_\alpha}$ denotes the complex conjugate of $f_{\alpha}$.
\end{theorem}

\section{$p$-adic integration on Hitchin systems}\label{section: p-adic integration on Hitchin systems}
The goal of this section is to apply the theory of $p$-adic integration discussed in Section \ref{subsection: p-adic integration and point counting} to Hitchin systems. The topological mirror symmetry equality in \cite{MWZ} Theorem 6.11 relies on the fact that for both the $\on{SL}_n$-side and $\on{PGL}_n$-side, there exists a large enough open subset that forms a torsor for certain commutative group scheme over the entire Hitchin base (see Remark \ref{remark: abstract Hitchin systems}). With an eye towards application in the parabolic setting, we proceed our discussion without assuming the existence of this open locus. The main conclusion we record here is Theorem \ref{main thm: p-adic integral for Hitchin}. 
\subsection{Weak Abstract Hitchin systems}\label{subsection: abstract hitchin systems}

\begin{definition}\label{definition: weak abstract Hitchin system}
Let $\mathcal{A}$ be a smooth $R$-variety. Let $\mathcal{M}$ be an admissible finite abelian group stack over $R$ together with a proper map $h: \mathcal{M}\longrightarrow \mathcal{A}$. We call $(\mathcal{M}, \mathcal{A})$ a \emph{weak abstract Hitchin system} if there exists an open dense subset $\mathcal{A}^0\subset \mathcal{A}$ and an abelian $\mathcal{A}^0$-scheme $\mathcal{P}$ such that the restriction $\mathcal{M}^0= \mathcal{M}\times _{\mathcal{A}}\mathcal{A}^0$ of $\mathcal{M}$ to $\mathcal{A}^0$ is a $\mathcal{P}$-torsor and $\on{codim}(\mathcal{M}\backslash\mathcal{M}^0)\geq 1$.
\end{definition}

\begin{remark}\label{remark: abstract Hitchin systems}
In \cite{MWZ} Definition 6.8, an \emph{abstract Hitchin system} is defined to be a weak abstract Hitchin system in Definition \ref{definition: weak abstract Hitchin system} with an open dense substack $\mathcal{M}'\subset \mathcal{M}$ that forms a torsor for a commutative group scheme $\mathcal{Q}$ over $\mathcal{A}$, with the assumption that $\mathcal{Q}\times_{\mathcal{A}}\mathcal{A}^0=\mathcal{P}$ and $\on{codim}(\mathcal{M}\backslash\mathcal{M}')\geq 2$.
\end{remark}

Let $(\mathcal{M}, \mathcal{A})$ be a weak abstract Hitchin system over $R$ such that $r$ is invertible in $R$, and let $\alpha$ be a $\mu_r$-gerbe on $\mathcal{M}$. We denote by $\on{Split}'(\mathcal{M}^0/\mathcal{A}^0, \alpha)$ the principal component of the space of relative
splittings of $\alpha$ on $\mathcal{M}^0$, see \cite{MWZ} Definition 6.4. We briefly recall the definition of $\on{Split}'(\mathcal{M}^0/\mathcal{A}^0, \alpha)$ as follows. We first associate with $\alpha$ its stack of relative splittings $\widetilde{\on{Split}}_{\mu_r}(\mathcal{M}^0/\mathcal{A}^0, \alpha)$: for any map of $R$-schemes $f: S\longrightarrow \mathcal{A}^0$, it classifies splittings of the $\mu_r$-gerbe $f^{*}\alpha$ on $S\times_{\mathcal{A}^0}\mathcal{M}^0$. The stack $\widetilde{\on{Split}}_{\mu_r}(\mathcal{M}^0/\mathcal{A}^0, \alpha)$ is naturally a pseudo torsor for the moduli stack $\widetilde{\on{Tor}}_{\mu_r}(\mathcal{M}^0/\mathcal{A}^0)$ of $\mu_r$-torsors. We denote by ${\on{Split}}_{\mu_r}(\mathcal{M}^0/\mathcal{A}^0, \alpha)$ the $\mu_r$-rigidification of $\widetilde{\on{Split}}_{\mu_r}(\mathcal{M}^0/\mathcal{A}^0, \alpha)$ in the sense of \cite{ACV}, Section 5, which becomes a pseudo torsor for the scheme $\on{Pic}(\mathcal{M}^0/\mathcal{A}^0)[r]$ of $r$-torsion points in the relative Picard scheme $\on{Pic}(\mathcal{M}^0/\mathcal{A}^0)$. Now we define \begin{equation}\label{eq; definition of spitting}
  \on{Split}'(\mathcal{M}^0/\mathcal{A}^0, \alpha)= \on{Split}_{\mu_r}(\mathcal{M}^0/\mathcal{A}^0, \alpha)\times^{\on{Pic}(\mathcal{M}^0/\mathcal{A}^0)[r]} \on{Pic}^{\tau}(\mathcal{M}^0/\mathcal{A}^0),  
\end{equation}
where $\on{Pic}^{\tau}(\mathcal{M}^0/\mathcal{A}^0)$ is the torsion component in $\on{Pic}(\mathcal{M}^0/\mathcal{A}^0)$. Since $\mathcal{M}^0$ is a $\mathcal{P}$-torsor over $\mathcal{A}^0$, we have  $\on{Pic}^{\tau}(\mathcal{M}^0/\mathcal{A}^0)\cong \mathcal{P}^{\vee}$, where $\mathcal{P}^{\vee}$ is the dual abelian scheme of $\mathcal{P}$. Therefore $\on{Split}'(\mathcal{M}^0/\mathcal{A}^0, \alpha)$ is a pseudo $\mathcal{P}^{\vee}$-torsor. 

\begin{definition}\label{def: dual Hitchin fibration}
We consider two weak abstract Hitchin systems $(\mathcal{M}_1, \mathcal{M}_2, \mathcal{A}$) over the same base $\mathcal{A}$, such that $\mathcal{M}_i$ becomes a $\mathcal{P}_i$-torsor when restricted to the same open dense subset $\mathcal{A}^0\subset \mathcal{A}$. Let $\alpha_i$ be a $\mu_r$-gerbe on $\mathcal{M}_i$ for $i=1, 2$, and we assume $r$ is invertible in $R$ and $R$ contains all the $r$-th roots of unity. We call $(\mathcal{M}_1, \mathcal{M}_2, \mathcal{A}, \alpha_1, \alpha_2)$ a \emph{dual pair} of weak abstract Hitchin systems if the following conditions hold.

(a) $\mathcal{P}_1$ and $\mathcal{P}_2$ are dual abelian schemes over $\mathcal{A}^0$, and there is an \'etale isogeny $\phi: \mathcal{P}_1\longrightarrow \mathcal{P}_2$ such that the degree of $\phi$ is invertible in $R$.

(b) there are isomorphisms of $\mathcal{P}_i$-torsors for $i=1, 2$
\[
    \on{Split}'(\mathcal{M}^0_1/\mathcal{A}^0, \alpha_1)\cong \mathcal{M}^0_2,    \quad  \on{Split}'(\mathcal{M}^0_2/\mathcal{A}^0, \alpha_2)\cong \mathcal{M}^0_1.
\]

(c) For any local field $F$, any homomorphism $R\to \mathcal{O}_F$ and any point $a\in \mathcal{A}(\mathcal{O}_F)\cap \mathcal{A}^0(F)$ with corresponding $a_F\in \mathcal{A}^0(F)$, 
if both fibres $(\mathcal{M}^0_{1})_{a_F}$ and $(\mathcal{M}^0_{2})_{a_F}$ have $F$-rational points, then both $\mathbb{G}_m$-gerbes induced from $\alpha_1|_{(\mathcal{M}^0_{1})_{a_F}}$ and $\alpha_2|_{(\mathcal{M}^0_{2})_{a_F}}$ split. 
\end{definition}

Now we consider a weak abstract Hitchin system $(\mathcal{M}, \mathcal{A})$ over $R$ such that when restricted to $\mathcal{A}^0\subseteq \mathcal{A}$, $\mathcal{M}^0$ is a torsor for an abelian $\mathcal{A}^0$-scheme $\mathcal{P}$. Since $\mathcal{P}\longrightarrow \mathcal{A}^0$ is smooth, $h: \mathcal{M}^0\longrightarrow \mathcal{A}^0$ is also smooth, therefore we have the following isomorphism of invertible sheaves
\begin{equation}\label{eq: top form yoga}
    \Omega^{\on{top}}_{\mathcal{M}^0/R}\cong h^{*}\Omega^{\on{top}}_{\mathcal{A}^0/R}\otimes \Omega^{\on{top}}_{\mathcal{M}^0/\mathcal{A}^0}.
\end{equation}
Since $\mathcal{P}$ is an abelian $\mathcal{A}^0$-scheme, there exist translation invariant trivializing sections of $\Omega^{\on{top}}_{\mathcal{P}/\mathcal{A}^0}$ (at least locally on $\mathcal{A}^0$). As observed in \cite{MWZ} Lemma 6.13, for any translation invariant global section $ \omega\in\Gamma(\mathcal{P}, \Omega^{\on{top}}_{\mathcal{P}/\mathcal{A}^0})$, one can associate with it a translation invariant global section of  $\Omega^{\on{top}}_{\mathcal{M}^0/\mathcal{A}^0}$ as follows. We pick some \'etale cover $U\to\mathcal{A}^0$ that trivializes the $\mathcal{P}$-torsor $\mathcal{M}^0$, i.e. there exists an isomorphism $\mathcal{M}^0_U\cong \mathcal{P}_U$ of $\mathcal{P}_U$-torsors. Through this isomorphism we get a section of $\Omega^{\on{top}}_{\mathcal{M}^0_U/U}$ from $\omega$. Since $\omega$ is translation invariant, this section of $\Omega^{\on{top}}_{\mathcal{M}^0_U/U}$ descends to a translation invariant section of $\Omega^{\on{top}}_{\mathcal{M}^0/\mathcal{A}^0}$, and the resulting section does not depend on the choice of the \'etale cover $U\to X$ and the trivialization.

\begin{lemma}\label{lemma: top form yoga}
The map from translation invariant relative volume forms on $\mathcal{P}/\mathcal{A}^0$ to translation invariant relative volume forms on $\mathcal{M}^0/\mathcal{A}^0$ described above induces an isomorphism 
\[
h_*\Omega^{\on{top}}_{\mathcal{P}/\mathcal{A}^0}\cong h_*\Omega^{\on{top}}_{\mathcal{M}^0/\mathcal{A}^0}
\]
of $\mathcal{O}_{\mathcal{A}^0}$-modules. 
\end{lemma}

\begin{proof}
    The statement follows from the fact that for any Zariski open subset $V\subseteq \mathcal{A}^0$, pulling back along $\mathcal{P}_V\longrightarrow V$ and $\mathcal{M}^0_V\longrightarrow V$ induce isomorphisms 
    \[
        \Gamma(V, \mathcal{O}_{V})\longrightarrow \Gamma(\mathcal{P}_V, \mathcal{O}_{\mathcal{P}_v}),\quad \Gamma(V, \mathcal{O}_{V})\longrightarrow \Gamma(\mathcal{M}^0_V, \mathcal{O}_{\mathcal{M}^0_V})
    \]
    of global functions. This is because both maps $\mathcal{P}\longrightarrow \mathcal{A}^0$ and $\mathcal{M}^0\longrightarrow \mathcal{A}^0$ are proper with geometrically reduced and geometrically connected fibers.
\end{proof}

For the rest of this paper, for any global section $\tilde{\omega}\in \Gamma(\mathcal{M}^0, \Omega^{\on{top}}_{\mathcal{M}^0/\mathcal{A}^0})$, we will denote by $\tilde{\omega}'$ the corresponding section in $\Gamma(\mathcal{P}, \Omega^{\on{top}}_{\mathcal{P}/\mathcal{A}^0})$, and vice versa. 

If we assume that the invertible sheaf $\Omega^{\on{top}}_{\mathcal{A}^0/R}$ is trivial and fix a trivializing section $\omega_{\mathcal{A}^0}$, it follows from isomorphism (\ref{eq: top form yoga}) and Lemma \ref{lemma: top form yoga} that any trivializing section $\omega\in \Gamma(\mathcal{M}^0, \Omega^{\on{top}}_{\mathcal{M}^0/R})$ can be written uniquely as
\begin{equation}\label{eq: top form fubini}  \omega=h^{*}\omega_{\mathcal{A}^0}\wedge \tilde{\omega},
\end{equation}
where $\tilde{\omega}\in \Gamma(\mathcal{M}^0, \Omega^{\on{top}}_{\mathcal{M}^0/\mathcal{A}^0})$ is a translation invariant trivializing section.

\subsection{$p$-adic integration and mirror symmetry}\label{subsec: $p$-adic integration and mirror symmetry}
In this subsection, we restrict ourselves to the case when $R=\mathcal{O}_F$, where $\mathcal{O}_F$ is the ring of integers of a non-Archimedean local field $F$. Let $(\mathcal{M}, \mathcal{A})$ be a weak abstract Hitchin system over $\mathcal{O}_F$ such that when restricted to the open subset $\mathcal{A}^0\subset \mathcal{A}$, $\mathcal{M}^0$ becomes a torsor for an abelian $\mathcal{A}^0$-scheme $\mathcal{P}$. Since we assume $\mathcal{M}$ is admissible, it admits a presentation $\mathcal{M}\cong [Y/\Gamma]$, where $Y$ is a smooth quasi-projective $\mathcal{O}_F$-variety with a generically free $\Gamma$-action. We denote by $\on{pr}: Y\longrightarrow M=Y/\Gamma$ the quotient map to the coarse moduli space $M$. We denote by $U\subset Y$ the locus where the $\Gamma$-action is free. Note that since $\mathcal{M}^0/\mathcal{A}^0$ is a $\mathcal{P}$-torsor, the open substack $\mathcal{M}^0\subset \mathcal{M}$ is actually a scheme over $\mathcal{O}_F$. It follows that $\mathcal{M}^0$ maps isomorphically to an open subscheme of $\on{pr}(U)$, which we still denote by $\mathcal{M}^0$.

\begin{definition}
        We define $\mathcal{A}(\mathcal{O}_F)^{\flat}= \mathcal{A}(\mathcal{O}_F)\cap \mathcal{A}^0(F)$ and $M(\mathcal{O}_F)^{\flat}= M(\mathcal{O}_F)\cap \mathcal{M}^0(F).$
\end{definition}

\begin{remark}\label{remark: importance of properness}
   (a) Recall that the orbifold measure $\mu_{orb}$ is defined over the $F$-analytic manifold $M(\mathcal{O}_F)^{\sharp}=M(\mathcal{O}_F)\cap \on{pr}(U)(F)$. Note that the complement of $M(\mathcal{O}_F)^{\flat}$ in $M(\mathcal{O}_F)^{\sharp}$ is $(\on{pr}(U)\backslash \mathcal{M}^0)(\mathcal{O}_F)$ which has measure zero. 
   
   (b) Since the morphism $h:\mathcal{M}^0\longrightarrow \mathcal{A}^0$ is smooth, it restricts to a submersion $h_{\flat}: M(\mathcal{O}_F)^{\flat}\longrightarrow \mathcal{A}(\mathcal{O}_F)^{\flat}$ of $F$-analytic manifolds. Since the map $h:M\longrightarrow \mathcal{A}$ is proper, for any $a\in \mathcal{A}(\mathcal{O}_F)^{\flat}$, we have $h_{\flat}^{-1}(a)=h^{-1}(a)(F)$. 
\end{remark}

Now we consider $(\mathcal{M}_1, \mathcal{M}_2, \mathcal{A}, \alpha_1, \alpha_2)$ a dual pair of weak abstract Hitchin systems over $\mathcal{O}_F$. For $i=1, 2$, when restricted to the open subset $\mathcal{A}^0\subset \mathcal{A}$, $\mathcal{M}^0_i$ becomes a torsor for an abelian $\mathcal{A}^0$-scheme $\mathcal{P}_i$, and we have an \'etale isogeny $\phi: \mathcal{P}_1\longrightarrow \mathcal{P}_2$. We denote by $f_{\alpha_i}: M_i(\mathcal{O}_F)^{\sharp}\longrightarrow \mathbb {C}$ the function corresponding to $\alpha_i$ as defined in (\ref{eq: f_gerbe}). We further assume that there is a trivializing section $\omega_{\mathcal{A}^0}$ of the invertible sheaf $\Omega^{\on{top}}_{\mathcal{A}^0/\mathcal{O}_F}$, and that there is a trivializing section $\omega_i$ of $\Omega^{\on{top}}_{\on{pr}(U_i)/\mathcal{O}_F}$ for $i=1, 2$. By the discussion at the end of Subsection \ref{subsection: abstract hitchin systems}, the top-degree form $\omega_i$ can be written uniquely as a wedge product when restricted to $\mathcal{M}_i^0$: 
\[
   \omega_{i}=h_i^{*}\omega_{\mathcal{A}^0}\wedge \tilde{\omega}_i, 
\]
where $\tilde {\omega}_i\in \Gamma(\mathcal{M}_i^0, \Omega^{\on{top}}_{\mathcal{M}_i^0/\mathcal{A}^0})$ is a translation invariant trivializing section. We denote by $\tilde{\omega}'_i$ the corresponding section in $\Gamma(\mathcal{P}_i, \Omega^{\on{top}}_{\mathcal{P}_i/\mathcal{A}^0})$, see Lemma \ref{lemma: top form yoga}. 

\begin{theorem}\label{main thm: p-adic integral for Hitchin}
If we assume $\phi^{*}\tilde{\omega}'_2= \tilde{\omega}'_1$, then we have the following equality of $p$-adic integrals
\[
\int_{M_1(\mathcal{O}_F)^{\sharp}}\bar{f}_{\alpha_1}|\omega_1|=\int_{M_2(\mathcal{O}_F)^{\sharp}}\bar{f}_{\alpha_2}|\omega_2|.
\]
\end{theorem}

\begin{proof}
Combining Proposition \ref{prop: properties of p-adic integration} and Remark \ref{remark: importance of properness}, we have 
\[
\int_{M_i(\mathcal{O}_F)^{\sharp}}\bar{f}_{\alpha_i}|\omega_i|=\int_{M_i(\mathcal{O}_F)^{\flat}}\bar{f}_{\alpha_i}|\omega_i|=\int_{\mathcal{A}(\mathcal{O}_F)^{\flat}\ni a}|\omega_{\mathcal{A}^0}|\int_{h_i^{-1}(a)(F)}\bar{f}_{\alpha_i}|(\tilde{\omega}_i)_{a_F}|
\]
for $i=1,2$. Now the theorem follows from \cite {GWZ} Proposition 2.2 and \cite{MWZ} Theorem 6.16.
\end{proof}

\begin{remark}\label{rk: p-adic integral for Hitchin}
Let $\mathcal{M}_i\cong [Y_i/\Gamma_i]$ be a presentation of $\mathcal{M}_i$ as in Definition \ref{def: admissible stack}, and let $\on{pr}: Y_i\longrightarrow M_i=Y_i/\Gamma_i$ be the quotient map. Besides the assumptions for Theorem \ref{main thm: p-adic integral for Hitchin}, if we further assume that for $i=1, 2$, the top-degree form $\on{pr}^{*}\omega_i$ extends to a $\Gamma_i$-invariant trivializing section of $\Omega^{\on{top}}_{Y_i/\mathcal{O}_F}$, then by Lemma \ref{lemma: orbifold measure special description}, the equality in Theorem \ref{main thm: p-adic integral for Hitchin} becomes
\[
    \int_{M_1(\mathcal{O}_F)^{\sharp}}\bar{f}_{\alpha_1}d\mu_{orb}=\int_{M_2(\mathcal{O}_F)^{\sharp}}\bar{f}_{\alpha_2}d\mu_{orb}.
\]
By Theorem \ref{thm: stringy point count=p-adic integral}, this equality implies the following equality of stringy point-counts of the corresponding special fibers
\[
    \#^{\alpha_1}_{st}((\mathcal{M}_1)_{k_F})= \#^{\alpha_2}_{st}((\mathcal{M}_2)_{k_F}).
\]
\end{remark}

\section{Moduli spaces of parabolic Higgs bundles}\label{section: moduli spaces of parabolic Higgs bundles}
\subsection{Moduli space of parabolic $\on{GL}_n$-Higgs bundles}
We fix a Noetherian integral scheme $\mathcal{B}$. We assume $\mathcal{B}$ is of finite type over a universally Japanese ring. The main examples that will be relevant to our applications are when $\mathcal{B}=\on{Spec}(\mathbb{C})$, $\mathcal{B}=\on{Spec}(R)$ where $R\subset \mathbb{C}$ is a finite generated $\mathbb{Z}$-subalgebra, $\mathcal{B}=\on{Spec}(\mathbb{F}_q)$ where $\mathbb{F}_q$ is a finite field and $\mathcal{B}=\on{Spec}(\mathbb{F}_q[[t]])$. Let $X$ be a smooth projective curve over $\mathcal{B}$, i.e. $X\longrightarrow \mathcal{B}$ is smooth projective of relative dimension one with geometrically connected fibers. We assume the genus of $X$ satisfies $g_X\geq 2$. We assume $X$ admits a marked point $q$, i.e. an element $q\in X(\mathcal{B})$. Let $K$ be the relative canonical bundle of $X$ over $\mathcal{B}$. We fix a line bundle $M$ on $X$ such that $M\otimes K^{-1}$ is either trivial or of strictly positive degree. We will consider parabolic Higgs bundles of rank $n$ and multiplicities $\textbf{\textit{m}}=(m_1, m_2,\dots, m_r)$, where $m_i$ are positive integers that sum up to $n$. We recall the following definitions concerning parabolic Higgs bundles. 
\begin{definition}
(a) An \emph{$M$-twisted quasi-parabolic Higgs bundle} is a triple $\textbf{\textit{E}}=(E, \phi, E_q^{\bullet})$, where $E$ is a vector bundle of rank $n$ on $X$, $\phi: E\longrightarrow E\otimes M(q)$ is an $\mathcal{O}_X$-linear map, and $E_q^{\bullet}$ is a partial flag structure of type $\textbf{\textit{m}}$ on the fiber $E_q$ of $E$ at $q\in X(\mathcal{B})$:
\[
    \{0\}=E_q^0\subset E_q^1\subset \cdots \subset E_q^r=E_q
\]
such that $\on{dim}(E_q^i/E_q^{i-1})=m_i$. We further require that when restricted to the fiber $E_q$, the Higgs field $\phi_q$ is nilpotent with respect to the partial flag structure, i.e. $\phi_q(E_q^i)\subset E_q^{i-1}$. 

(b) An \emph{$M$-twisted parabolic Higgs bundle} is an $M$-twisted quasi-parabolic Higgs bundle $\textbf{\textit{E}}$ together with \emph{parabolic weights} $\boldsymbol{\alpha}=(\alpha_0, \alpha_1, \dots, \alpha_r)$, which is a tuple of real numbers that satisfies
\[
    1=\alpha_0>\alpha_{1}>\cdots>\alpha_r\geq0.
\]
We define the \emph{parabolic degree} of an $M$-twisted parabolic Higgs bundle to be 
\[
\on{pdeg}(\textbf{\textit{E}})=\on{deg}(E)+\sum_{i=1}^{r}\alpha_im_i.
\]
\end{definition}

Let $F$ be a subbundle of $E$. The partial flag structure $E_q^\bullet$ on $E_q$ naturally induces a partial flag structure $F_q^{\bullet}$ on $F_q$. The induced parabolic weights for $(F, F_q^{\bullet})$ is defined by $\alpha'_i=\on{max}\{\alpha_j|F_q^i\subseteq E_q^j\}$. If $F$ is further assumed to be $\phi$-invariant, we get a parabolic Higgs subbundle $\textbf{\textit{F}}$ of $\textbf{\textit{E}}$. Similarly, the partial flag structure on $E_q$ induces a partial flag structure on the quotient $(E/F)_q$ with parabolic weights $\alpha''_i=\on{max}\{\alpha_j|(E/F)_q^i=E^j_q/F_q\}$. We have the following equality of parabolic degrees
\[
    \on{pdeg}(\textbf{\textit{E}})=\on{pdeg}(\textbf{\textit{F}})+\on{pdeg}(\textbf{\textit{E}}/\textbf{\textit{F}}).
\]

\begin{definition}
(a) An $M$-twisted parabolic Higgs bundle $\textbf{\textit{E}}$ is \emph{semistable} if for all $\phi$-invariant proper non-zero subbundle $F\subset E$ we have
\[
    \frac{\on{pdeg}(\textbf{\textit{F}})}{\on{rk}(\textbf{\textit{F}})}\leq \frac{\on{pdeg}(\textbf{\textit{E}})}{\on{rk}(\textbf{\textit{E}})}.
\]  
It is called \emph{stable} if the inequality is always strict. 

(b) We denote by $\mathcal{M}_{\on{GL}_n}^{ \boldsymbol{\alpha}}$ the moduli space of semistable $M$-twisted parabolic Higgs bundles with parabolic weights $\boldsymbol{\alpha}$. We denote by $\mathcal{M}_{\on{GL}_n}^{d, \boldsymbol{\alpha}}$ the degree $d$ component of $\mathcal{M}_{\on{GL}_n}^{ \boldsymbol{\alpha}}$. Each $\mathcal{M}_{\on{GL}_n}^{d, \boldsymbol{\alpha}}$ is a quasi-projective variety over $\mathcal{B}$, see \cite{Y},  Corollary 1.6 and Corollary 4.7. 
\end{definition}

 We say $\boldsymbol{\alpha}$ is \emph{generic} for degree $d$ if there exists no parabolic vector bundle with degree $d$ and parabolic weights $\boldsymbol{\alpha}$ that is semistable but not stable. More precisely, we require that the equality
 \begin{equation}\label{eq: generic weights}
     \frac{d'+\sum\alpha_j'm_j'} {n'}=\frac{d+\sum \alpha_i m_i}{n}
 \end{equation}
 doesn't have any solutions, where  $(\boldsymbol{m}', \boldsymbol{\alpha}')$ are the parabolic weights and multiplicities induced from a proper non-zero subbundle $F\subset E$ of rank $n'$ and degree $d'$. For a set of parabolic weights $\boldsymbol{\alpha}$ generic for degree $d$, it is shown in \cite{Y} that the moduli space $\mathcal{M}_{\on{GL}_n}^{d, \boldsymbol{\alpha}}$ is a smooth quasi-projective variety over $\mathcal{B}$. 
 
 The following lemma gives a criterion for the existence of generic parabolic weights for a given degree $d$. This criterion will be used in the proof of Theorem \ref{thm: main theorem independence}. 
 
 \begin{lemma}\label{lemma: criterion for generic weights}
There exists a set of parabolic weights $\boldsymbol{\alpha}$ generic for degree $d$ if and only if $\on{gcd}(m_1, m_2, \dots, m_r,d)=1$. 
 \end{lemma}
 \begin{proof}
 Let $V$ be a vector space of dimension $n$ with a partial flag structure of type $\boldsymbol{m}$. For different choices of integer $d'$ and subspace $V'\subset V$ of dimension $n'$, (\ref{eq: generic weights}) imposes countably many restrictions on $\boldsymbol{\alpha}=(\alpha_1, \alpha_2, \dots, \alpha_r)$. Since all $\alpha_i$ take values from real numbers, there exists no generic parabolic weights for degree $d$ if and only if (\ref{eq: generic weights}) is automatically satisfied for some choices of $d'$ and $V'$. In other words, there exists a generic $\boldsymbol{\alpha}$ for degree $d$ if and only if the following family of equations 
 \begin{equation}\label{eq: generic weights existence}
     \frac{d}{n}=\frac{d'}{n'},  \frac{m_i}{n}=\frac{m_i'}{n'}, i=1,2,\dots,r
 \end{equation}
 has no solutions, where $d'$ is an integer, $n'$ is a positive integer smaller than $n$, and $\boldsymbol{m}'=(m'_1, m'_2, \dots m'_s)$ is a partition of $n'$. 
 
 Now we prove the ``if" direction of the lemma. Suppose $\on{gcd}(m_1, m_2, \dots, m_r,d)=1$. If there exists a set of solutions $(d', n', \boldsymbol{m'})$ for (\ref{eq: generic weights existence}), then we have $n|dn'$ and $n|m_in'$ for each $i$. Since $\on{gcd}(m_1, m_2, \dots, m_r,d)=1$, we have $n|n'$, which contradicts our assumption that $n'$ is a positive integer smaller than $n$. 
 
 For the other direction, suppose $\on{gcd}(m_1, m_2, \dots, m_s,d)=c>1$, then $d'=d/c$, $n'=n/c$ and $m_i'=m_i/c$ gives a set of solutions for (\ref{eq: generic weights existence}). 
 \end{proof}

Now we state the first main theorem of the paper, which we prove in Section \ref{section: proof of mirror symmetry}.  
\begin{theorem}\label{thm: main theorem independence}
Let $X$ be a smooth projective curve over a field $k$. Let $\boldsymbol{\alpha}$ be a set of parabolic weights generic for degree $d$. 

(a) Let $k=\mathbb{C}$. Then the $E$-polynomial $E(\mathcal{M}_{\on{GL}_n}^{d,\boldsymbol{\alpha}};u,v)$ is independent of $d$ and $\boldsymbol{\alpha}$.

(b) Let $k=\mathbb{F}_q$ such that $char(k)> n$. Then the point-count $\#\mathcal{M}_{\on{GL}_n}^{d,\boldsymbol{\alpha}}(k)$ is independent of $d$ and $\boldsymbol{\alpha}$. 
\end{theorem}

By taking the characteristic polynomial of the Higgs field, we get the parabolic Hitchin map
\[
h: \mathcal{M}_{\on{GL}_n}^{ \boldsymbol{\alpha}}\longrightarrow \mathcal{A}_{\on{GL}_n}, 
\]
where $\mathcal{A}_{\on{GL}_n}=\displaystyle{ \bigoplus_{i=1}^n }\Gamma(X, M(q)^i)$\footnote{More precisely, $\mathcal{A}_{\on{GL}_n}$ is the scheme of sections of $\displaystyle{ \bigoplus_{i=1}^n }\Gamma(X, M(q)^i)$, which is an affine $\mathcal{B}$-space.}. The map $h$ satisfies the following property:

\begin{theorem}[cf. \cite{Y} Corollary 1.6 and Corollary 5.12]\label{thm: Hitchin map proper} The parabolic Hitchin map $h$ is proper. 
\end{theorem}

The (scheme-theoretic) image of $h$ is an affine subspace $\mathcal{A}_{P}\subset \mathcal{A}_{\on{GL}_n}$ which depends on the parabolic multiplicities $\textbf{\textit{m}}$. In order to describe $\mathcal{A}_P$, we introduce the following notations. Let 
\[\boldsymbol{\lambda}=(\lambda_1, \lambda_2,\dots, \lambda_s), \lambda_1\geq \lambda_2, \dots, \geq \lambda_s\] 
be the partition of $n$ conjugate to $\textbf{\textit{m}}$. We define $\lambda_0=0$. Let $\textbf{\textit{t}}=(t_1, t_2,\dots t_n)$ be the ordered $n$-tuple of integers defined by $t_i=j$ if and only if $\displaystyle\sum_{k=0}^{j-1}\lambda_k<i\leq \displaystyle\sum_{k=0}^{j}\lambda_k$, i.e. the first $\lambda_1$ elements are $1$, the next $\lambda_2$ elements are $2$,$\dots$, and the last $\lambda_s$ elements are $s$. It is shown in \cite{BK} that
\begin{equation}\label{eq: Hitchin base}
   \mathcal{A}_P=\displaystyle{ \bigoplus_{i=1}^n }\Gamma(X, M^i((i-t_i)q)). 
\end{equation}

For the rest of this section, we assume the base scheme $\mathcal{B}$ satisfies the property that for any algebraically closed field $\bar{k}$ such that $\mathcal{B}(\bar{k})$ is non-empty, we have $\on{char}(\bar{k})=0$ or $\on{char}(\bar{k})>n$. The generic fiber of the parabolic Hitchin map is described in the following theorem.

\begin{theorem}[cf. \cite{shen} Theorem 2.11.]\label{theorem: spectral data GLn}
There exists an open dense subset $\mathcal{A}_P^0\subset \mathcal{A}_P$ and a flat family of smooth projective
curves $\Sigma\longrightarrow \mathcal{A}_P^0$ such that
\[
    \mathcal{M}_{\on{GL}_n}^{ \boldsymbol{\alpha}}\times_{\mathcal{A}_P}\mathcal{A}_P^0\cong \on{Pic}(\Sigma/\mathcal{A}_P^0),
\]
where $\on{Pic}(\Sigma/\mathcal{A}_P^0)$ is the relative Picard scheme of $\Sigma$ over $\mathcal{A}_P^0$. This family of curves admits a map $p_X: \Sigma\longrightarrow X$ such that the corresponding $\Sigma\longrightarrow X\times \mathcal{A}_P^0$ is finite, and the isomorphism above is induced by pushing-forward along $p_X$. 
\end{theorem}

\begin{remark}\label{rk: singular spectral curves}
When $\textbf{\textit{m}}=(1,1,\dots,1)$, i.e. we consider complete flags at the marked point, $\mathcal{A}_P^0$ is the locus where the spectral curves are smooth. In all the other cases, the spectral curves are always singular above the marked point $q$. For any geometric point $a\in \mathcal{A}_P^0(\bar{k})$, the spectral curve at $a$ locally looks like 
\[
    \on{Spec}k[[x,y]]/(\prod_{i=1}^s(y^{\lambda_i}-a_ix)), a_i\in k[[x,y]]^{\times}
\]
above $q$, and $\Sigma_a$ is the normalization of this spectral curve. Since $\mathcal{A}_P^0$ lies in the locus where the spectral curves are integral, $\mathcal{M}_{\on{GL}_n}^{ \boldsymbol{\alpha}}\times_{\mathcal{A}_P}\mathcal{A}_P^0$ lies in the stable locus of $\mathcal{M}_{\on{GL}_n}^{ \boldsymbol{\alpha}}$.
\end{remark}

\begin{remark}
The degree of a parabolic Higgs bundle and the degree of its corresponding spectral sheaf on $\Sigma_a$ differ by the degree of the line bundle $\on{det}((p_X)_*\mathcal{O}_{\Sigma_{a}})$. The line bundle $\on{det}((p_X)_*\mathcal{O}_{\Sigma})$ on $X\times \mathcal{A}$ is isomorphic to the pull-back of a line bundle on $X$, which by abuse of notation we still denote by $\on{det}((p_X)_*\mathcal{O}_{\Sigma})$. We denote $d_{\boldsymbol{m}}=\on{deg}(\on{det}((p_X)_*\mathcal{O}_\Sigma))$ which depends on the parabolic multiplicities $\boldsymbol{m}$. It follows from the construction of $\Sigma$ that \[d_{\boldsymbol{m}}=-\displaystyle\frac{1}{2}n(n-1)(\on{deg}M+1)+\displaystyle\sum_{i=1}^{r}\frac{1}{2}m_i(m_i-1).
\]
\end{remark}

\begin{prop}\label{prop: picard injective}
For any geometric point $a\in \mathcal{A}^0_P(\bar{k})$, the pull-back map $p_X^{*}: \on{Pic}(X_{\bar{k}})\longrightarrow \on{Pic}(\Sigma_a)$ is injective. 
\end{prop}
We would like to apply the same argument as in \cite{BNR} Remark 3.10. For that purpose, we need the following lemma:
\begin{lemma}\label{lemma: split trick Higgs bundle}
For any $a\in \mathcal{A}^0_P(\bar{k})$, there exists a stable parabolic Higgs bundle $\textbf{\textit{E}}$ such that $h(\textbf{\textit{E}})=a$ and the underlying vector bundle $E$ decomposes as a direct sum of line bundles of distinct degree. 
\end{lemma}
\begin{proof}
We write $a=(a_1, a_2,\dots, a_n)$, $a_i\in \Gamma(X_{\bar{k}}, M^i((i-t_i)q))$. We consider the following matrix 
\[
  \phi=
    \left(
      \begin{array}{*6{c}}
      0 & 0&\cdots        &  0 &-a_n\\
       1& 0& \cdots & 0 &  -a_{n-1}\\
       0 &    1           &0&0& -a_{n-2}\\
       \vdots&\vdots&\vdots&\vdots&\vdots\\
          0&  0      &0&    1 & -a_1\\
          
      \end{array}
    \right)
\]
acting on $E=(L_n,L_{n-1},\dots,L_1)^t$, where $L_i=M^i((i-t_i)q)$. When restricted to the fiber $E_q$, the Higgs field $\phi_q$ is a matrix with Jordan blocks of type $\boldsymbol{\lambda}$, which lies in the Richardson orbit corresponding to the parabolic multiplicities  $\boldsymbol{m}$. It follows that there exists a (unique) partial flag structure of type $\boldsymbol{m}$ on $E_q$ that is compatible with $\phi_q$.

\end{proof}
\begin{proof}[Proof of Proposition \ref{prop: picard injective}]
It follows from Theorem \ref{theorem: spectral data GLn} that there exists a line bundle $H\in\on{Pic}(\Sigma_a)$ such that $(p_X)_{*}H\cong E=\displaystyle\bigoplus_{i=1}^{n}L_i$ defined in the proof of Lemma \ref{lemma: split trick Higgs bundle}. Let $N$ be a line bundle on $X_{\bar{k}}$ such that $p_X^{*}N\cong \mathcal{O}_{\Sigma_a}$. It follows from projection formula that $(p_X)_{*}(p_X^{*}N\otimes H)\cong N\otimes (p_X)_{*}H$, which implies $N\otimes \displaystyle\bigoplus_{i=1}^{n}L_i\cong \displaystyle\bigoplus_{i=1}^{n}L_i$. Since $L_i$ are line bundles of distinct degree, $N$ must be the trivial line bundle on $X_{\bar{k}}$. 
\end{proof}

\subsection{The $\on{SL}_n$-side and the $\on{PGL}_n$-side}\label{Subsection: sln vs pgln}
In this subsection, we introduce the moduli space of parabolic $\on{SL}_n$- and $\on{PGL}_n$-Higgs bundles. 
\begin{definition}
(a) Let $L$ be a line bundle on $X$. A $M$-twisted parabolic $\on{SL}_n$-Higgs bundle of determinant $L$ is a pair $(\textbf{\textit{E}}, \varphi)$, where $\textbf{\textit{E}}$ is a $M$-twisted parabolic Higgs bundle with trace zero Higgs field, and $\varphi: \on{det}(E)\cong L$ is an isomorphism of line bundles. 

(b) We denote by $\mathcal{M}_{\on{SL}_n}^{L, \boldsymbol{\alpha}}$ (resp. $\mathfrak{M}_{\on{SL}_n}^{L, \boldsymbol{\alpha}})$ the moduli space (resp. moduli stack) of semistable $M$-twisted parabolic $\on{SL}_n$-Higgs bundles of determinant $L$ and parabolic weights $\boldsymbol{\alpha}$.
\end{definition}

\begin{remark}
(a) Let $\boldsymbol{\alpha}$ be generic parabolic weights for degree $\on{deg}(L)$. Since the only automorphisms that a stable parabolic $\on{SL}_n$-Higgs bundle admits is multiplying by the $n$-th roots of unity, the natural map ${\mathfrak{M}}_{\on{SL}_n}^{L, \boldsymbol{\alpha}}\longrightarrow\mathcal{M}_{\on{SL}_n}^{L, \boldsymbol{\alpha}}$ gives the moduli stack ${\mathfrak{M}}_{\on{SL}_n}^{L, \boldsymbol{\alpha}}$ the structure of a $\mu_n$-gerbe over $\mathcal{M}_{\on{SL}_n}^{L, \boldsymbol{\alpha}}$. We denote this $\mu_n$-gerbe by $\check{\alpha}$.

(b) Let $\Gamma =\on{Pic}^0(X)[n]$ be the group scheme of $n$-torsion points in the Picard scheme $\on{Pic}(X)$. The moduli space $\mathcal{M}_{\on{SL}_n}^{L, \boldsymbol{\alpha}}$ admits a natural $\Gamma$-action. For each $\gamma\in \Gamma$, this action is defined by mapping the underlying vector bundle $E$ to $E\otimes \gamma$ and modifying the other data accordingly. 

(c) Consider the case when $X$ is a smooth complex projective curve. Let $L_1$ and $L_2$ be two line bundles on $X$ of the same degree $d$. Tensoring with an $n$-th root of $L_1^{-1}L_2$ induces an isomorphism $\mathcal{M}_{\on{SL}_n}^{L_1, \boldsymbol{\alpha}}\xrightarrow[]{\simeq}\mathcal{M}_{\on{SL}_n}^{L_2, \boldsymbol{\alpha}}$.
\end{remark}

Now we consider the $\on{SL}_n$-version of the parabolic Hitchin map
\[
    \check{h}: \mathcal{M}_{\on{SL}_n}^{L, \boldsymbol{\alpha}}\longrightarrow \mathcal{A},
\]
where $\mathcal{A}\subset \mathcal{A}_P$ is the affine subspace defined by the factor in $\Gamma(X, M(q))$ being zero. Let $\mathcal{A}^0=\mathcal{A}\cap \mathcal{A}_P^0$ and let $\Sigma/\mathcal{A}^0$ be the family of curves that appears in Theorem \ref{theorem: spectral data GLn} with $p_X: \Sigma\longrightarrow X$. In order to describe the generic fiber of $\check{h}$, we consider the norm map
\[
    \on{Nm}_{\Sigma/X\times \mathcal{A}^0}: \on{Pic}(\Sigma/\mathcal{A}^0)\longrightarrow \on{Pic}(X\times \mathcal{A}^0/\mathcal{A}^0)
\]  
defined by 
\[
     \on{Nm}_{\Sigma/X\times \mathcal{A}^0}(M)=\on{det}((p_X)_*M)\otimes \on{det}^{-1}((p_X)_*\mathcal{O}_\Sigma)
\]
for any $M\in\on{Pic}(\Sigma/\mathcal{A})$.

\begin{definition}
Let $\check{\mathcal{P}}=\on{Nm}_{\Sigma/X\times \mathcal{A}^0}^{-1}(\mathcal{O}_X)$. For a line bundle $L$ on $X$, let $\check{\mathcal{P}}_L=\on{Nm}_{\Sigma/X\times \mathcal{A}^0}^{-1}(L)$.
\end{definition}

This commutative group scheme $\check{\mathcal{P}}$ is the so-called \emph{relative Prym scheme} of $\Sigma/X\times \mathcal{A}^0$. We have the following description of the parabolic Hitchin fiber above $\mathcal{A}^0$.
\begin{prop}\label{prop: spectral data sln}
(a) $(\mathcal{M}_{\on{SL}_n}^{L, \boldsymbol{\alpha}})^0\coloneqq\mathcal{M}_{\on{SL}_n}^{L, \boldsymbol{\alpha}}\times_{\mathcal{A}}\mathcal{A}^0$ is isomorphic to $\check{\mathcal{P}}_{L\otimes \on{det}^{-1}((p_X)_*\mathcal{O}_\Sigma)}$,

(b) $\check{\mathcal{P}}$ is an abelian scheme over $\mathcal{A}^0$, and $(\mathcal{M}_{\on{SL}_n}^{L, \boldsymbol{\alpha}})^0$ is a $\check{\mathcal{P}}$-torsor. 
\end{prop}
\begin{proof}
Part (a) follows from Theorem \ref{theorem: spectral data GLn}. For part (b), the only thing left to check is that $\check{\mathcal{P}}$ is geometrically connected over $\mathcal{A}^0$. Note that $\check{\mathcal{P}}$ is the kernel of the surjective homomorphism
\[
\on{Pic}^0(\Sigma/\mathcal{A}^0)\xrightarrow{\on{Nm}}\on{Pic}^0(X\times \mathcal{A}^0/\mathcal{A}^0)
\]
of abelian schemes over $\mathcal{A}^0$. The dual morphism is given by the pull-back map
\[\on{Pic}^0(X\times \mathcal{A}^0/\mathcal{A}^0)\xrightarrow{p_X^{*}}\on{Pic}^0(\Sigma/\mathcal{A}^0).
\]
The desired statement that $\check{\mathcal{P}}$ is geometrically connected over $\mathcal{A}^0$ follows from Proposition \ref{prop: picard injective}, see \cite{mumford} Section 15, Theorem 1. 
\end{proof}

Now we describe the $\on{PGL}_n$-side of the Hitchin fibration.  
\begin{definition}
    We denote by $\mathcal{M}^{d, \boldsymbol{\alpha}}_{\on{PGL}_n}$ the moduli stack of $M$-twisted quasi-parabolic $\on{PGL}_n$-Higgs bundles that admits a presentation as a semistable parabolic $\on{GL}_n$-Higgs bundle of degree $d$ and parabolic weights $\boldsymbol{\alpha}$ on each geometric fiber of $X/\mathcal{B}$. 
\end{definition}

\begin{remark}
    (a) For each line bundle $L$ on $X$ of degree $d$, we have an isomorphism \begin{equation}\label{eq: sln vs pgln}
        \mathcal{M}^{d, \boldsymbol{\alpha}}_{\on{PGL}_n}\cong [  (\mathcal{M}^{d, \boldsymbol{\alpha}}_{\on{GL}_n})^{\on{tr=0}}/\on{Pic}^0(X)]\cong [\mathcal{M}^{L, \boldsymbol{\alpha}}_{\on{SL}_n}/\Gamma],
    \end{equation}
where $(\mathcal{M}^{d, \boldsymbol{\alpha}}_{\on{GL}_n})^{\on{tr=0}}\subseteq \mathcal{M}^{d, \boldsymbol{\alpha}}_{\on{GL}_n}$ is the subspace characterized by the Higgs field having zero trace. 

(b) The marked point $q\in X(\mathcal{B})$ induces a lifting of the $\Gamma$-action from $\mathcal{M}^{L, \boldsymbol{\alpha}}_{\on{SL}_n}$ to the moduli stack $\mathfrak{M}^{L, \boldsymbol{\alpha}}_{\on{SL}_n}$, see Example \ref{example: picard stack}. It follows that the $\mu_n$-gerbe $\check{\alpha}$ on $\mathcal{M}^{L, \boldsymbol{\alpha}}_{\on{SL}_n}$ descends to a $\mu_n$-gerbe on $\mathcal{M}^{d, \boldsymbol{\alpha}}_{\on{PGL}_n}$. Let $L_d$ be the degree $d$ line bundle defined by $L_d=\mathcal{O}_X(d'q)\otimes\on{det}((p_X)_*\mathcal{O}_\Sigma)$, where $d'=d-d_{\boldsymbol{m}}$. We denote by $\hat{\alpha}$ the $\mu_n$-gerbe on $\mathcal{M}^{d, \boldsymbol{\alpha}}_{\on{PGL}_n}$ that comes from the $\mu_n$-gerbe $\check{\alpha}$ on $\mathcal{M}^{L_d, \boldsymbol{\alpha}}_{\on{SL}_n}$. 
    
(c) Via isomorphisms (\ref{eq: sln vs pgln}), the $\on{SL}_n$-version of the parabolic Hitchin map descends to the $\on{PGL}_n$-version of the parabolic Hitchin map
\[
    \hat{h}: \mathcal{M}^{d, \boldsymbol{\alpha}}_{\on{PGL}_n}\longrightarrow \mathcal{A}. 
\]
\end{remark}
Now we describe the fiber of $\hat{h}$ over $\mathcal{A}^0\subset \mathcal{A}$. Notice that the action of $\on{Pic}(X)$ on $\on{Pic}(\Sigma/\mathcal{A}^0)$ is free by Proposition \ref{prop: picard injective}. We fix the following notations:
\begin{definition}
    Let $\hat{\mathcal{P}}=\on{Pic}^0(\Sigma/\mathcal{A}^0)/\on{Pic}^0(X)$ and $\hat{\mathcal{P}}_d=\on{Pic}^d(\Sigma/\mathcal{A}^0)/\on{Pic}^0(X)$.
\end{definition}

Recall that $\check{\mathcal{P}}$ is an abelian scheme over $\mathcal{A}^0$ that lies in the following short exact sequence
\[
0\xrightarrow{}\check{\mathcal{P}}\xrightarrow{}\on{Pic}^0(\Sigma/\mathcal{A}^0)\xrightarrow{\on{Nm}}\on{Pic}^0(X\times \mathcal{A}^0/\mathcal{A}^0)\xrightarrow{}0.
\]
The dual short exact sequence is given by 
\[
0\longrightarrow\on{Pic}^0(X\times \mathcal{A}^0/\mathcal{A}^0)\xrightarrow{p_X^{*}}\on{Pic}^0(\Sigma/\mathcal{A}^0)\longrightarrow \check{\mathcal{P}}^{\vee}\longrightarrow 0,
\]
therefore $\hat{\mathcal{P}}$ is an abelian scheme over $\mathcal{A}^0$ that is dual to $\check{\mathcal{P}}$. Combining with Proposition \ref{prop: spectral data sln}, we have

\begin{prop}
(a) $(\mathcal{M}^{d, \boldsymbol{\alpha}}_{\on{PGL}_n})^0\coloneqq\mathcal{M}^{d, \boldsymbol{\alpha}}_{\on{PGL}_n}\times_{\mathcal{A}} \mathcal{A}^0$ is isomorphic to $\hat{\mathcal{P}}_{d'}$, where $d'=d-d_{\boldsymbol{m}}$.

(b) $\hat{\mathcal{P}}$ is an abelian scheme over $\mathcal{A}^0$ that is dual to $\check{\mathcal{P}}$, and $(\mathcal{M}_{\on{PGL}_n}^{d, \boldsymbol{\alpha}})^0$ is a $\hat{\mathcal{P}}$-torsor. 
\end{prop}

Now we assume the base scheme $\mathcal{B}=\on{Spec}(R)$ is an affine integral scheme such that $n$ is invertible in $R$ and $R$ contains all the $n^{2g}$-th roots of unity. We also assume the finite group scheme $\on{Pic}^0(X)[n]$ is constant over $\mathcal{B}$. Recall that for any $\mu_n$-gerbe $\alpha$ and integer $d$, we can construct a new $\mu_n$-gerbe $\alpha^d$ by taking the induced torsor for the map $[d]: B\mu_n\to B\mu_n$ defined by mapping a line bundle to its $d$-th power. We have the following proposition which we prove in Section \ref{subsection: duality between gerbes}.

\begin{prop}\label{prop: dual Hitchin system}
Let $\boldsymbol{\alpha}_1$ (resp. $\boldsymbol{\alpha}_2$) be a set of parabolic weights that is generic for degree $d$ (resp. $e$). Then the quintuple $(\mathcal{M}^{L_d, \boldsymbol{\alpha}_1}_{\on{SL}_n}, \mathcal{M}^{e, \boldsymbol{\alpha}_2}_{\on{PGL}_n},\mathcal{A}, \check{\alpha}^{e'}, \hat{\alpha}^{d'})$ is a dual pair of weak abstract Hitchin systems over $\mathcal{B}$ in the sense of Definition \ref{def: dual Hitchin fibration}. Here  $d'=d-d_{\boldsymbol{m}}$, $e'=e-d_{\boldsymbol{m}}$ and $L_d$ is the degree $d$ line bundle defined by $L_d=\mathcal{O}_X(d'q)\otimes\on{det}((p_X)_*\mathcal{O}_\Sigma)$.
\end{prop}

Applying $p$-adic integration to this dual pair of weak abstract Hitchin systems leads to the second main theorem of this paper, which we prove in Section \ref{section: proof of mirror symmetry}. 

\begin{theorem}\label{thm: main theorem}
Let $X$ be a smooth complex projective curve. Let $L$ be a line bundle on $X$ of degree $d$ and let $\boldsymbol{\alpha}_1$ (resp. $\boldsymbol{\alpha}_2$) be a set of parabolic weights that is generic for degree $d$ (resp. $e$). Then we have the following equality of $E$-polynomials:
\[
 E(\mathcal{M}^{L, \boldsymbol{\alpha}_1}_{\on{SL}_n}; u,v)=E_{st}(\mathcal{M}^{e, \boldsymbol{\alpha}_2}_{\on{PGL}_n}, \hat{\alpha}^{d'}; u,v),
\]  
where $d'=d-d_{\boldsymbol{m}}$. 
\end{theorem}

\subsection{Torsor-gerbe duality for Hitchin systems}\label{subsection: duality between gerbes}
In this subsection, we discuss the duality relation between the two $\mu_n$-gerbes $\check{\alpha}$ and $\hat{\alpha}$ on the moduli space of parabolic $\on{SL}_n$- and $\on{PGL}_n$-Higgs bundles. Our discussion uses the framework of duality for certain ``good" commutative stacks. We refer the readers to \cite{Arinkin} and \cite{BB} for a detailed account of this theory. Let $\mathcal{B}$ be a Noetherian scheme such that $n$ is invertible on $\mathcal{B}$. Let $\mathcal{G}$ be a commutative group stack locally of finite type over $\mathcal{B}$. The dual commutative group stack $\mathcal{G}^{\vee}=\on{Hom}(\mathcal{G}, B\mathbb{G}_m)$ classifies 1-morphisms of group stacks from $\mathcal{G}$ to $B\mathbb{G}_m$. The main examples we are going to consider are:

\begin{example}\label{ex: commutative group stacks}
\mbox{}
    \begin{enumerate}
        \item $\mathcal{G}=\mathbb{Z}$, $\mathcal{G}^{\vee}=B\mathbb{G}_m$,
        \item $\mathcal{G}=B\mathbb{G}_m$, $\mathcal{G}^{\vee}=\mathbb{Z}$, 
        \item $\mathcal{G}=\mathbb{Z}_n$,  $\mathcal{G}^{\vee}=B\mu_n$. Here $\mu_n=\on{Spec}(\mathbb{Z}[x]/(x^n-1))$,
        \item $\mathcal{G}=\mu_n$,  $\mathcal{G}^{\vee}=B\mathbb{Z}_n$,
        \item $\mathcal{G}=\mathcal{P}$ is an abelian scheme, then $\mathcal{G}^{\vee}=\mathcal{P}^{\vee}$ is the dual abelian scheme. 
    \end{enumerate}
\end{example}
In fact, all the commutative group stacks we are going to consider are \'etale locally isomorphic to a finite product of commutative group stacks that appear in the Example \ref{ex: commutative group stacks} above. 

\begin{example}\label{example: picard stack}
Let $X$ be a smooth projective curve over $\mathcal{B}$. We denote by $\mathfrak{Pic}(X)$ (resp. $\on{Pic}(X)$) the Picard stack (resp. Picard scheme) of $X$. The natural map $\mathfrak{Pic}(X)\longrightarrow \on{Pic}(X)$ gives $\mathfrak{Pic}(X)$ the structure of a $\mathbb{G}_m$-gerbe over $\on{Pic}(X)$. If we further assume there exists a marked point $q\in X(\mathcal{B})$, then $\on{Pic}(X)$ can be identified with the moduli stack of line bundles on $X$ with a trivialization at $q$. This identification induces a splitting $\mathfrak{Pic}(X)\cong\on{Pic}(X)\times B\mathbb{G}_m$ of the $\mathbb{G}_m$-gerbe. Since the marked point $q$ also induces an isomorphism  $\on{Pic}(X)\cong\on{Pic}^0(X)\times \mathbb{Z}$, we have $\mathfrak{Pic}(X)\cong\on{Pic}^0(X)\times B\mathbb{G}_m\times \mathbb{Z}$. Note that a marked point always exists \'etale locally. Since $\on{Pic}^0(X)$ is a self-dual abelian scheme and ($B\mathbb{G}_m$, $\mathbb{Z}$) are dual to each other as commutative group stacks, we have an isomorphism $\mathfrak{Pic}(X)^{\vee}\cong \mathfrak{Pic}(X)$. It is well-known that this isomorphism does not depend on the choice of the marked point, hence it can be glued to a global isomorphism. We include a proof of this fact for completion of the exposition. 
\end{example}

\begin{lemma}
    The isomorphism $\mathfrak{Pic}(X)^{\vee}\cong \mathfrak{Pic}(X)$ in Example \ref{example: picard stack} does not depend on the choice of the marked point $q\in X(\mathcal{B})$. 
\end{lemma}
\begin{proof}
    Let $q_1$ and $q_2$ be two marked points of $X$. Let $\rho_1$ (resp. $\rho_2$) be the isomorphisms between $\mathfrak{Pic}(X)$ and $\on{Pic}^0(X)\times B\mathbb{G}_m\times \mathbb{Z}$ described in Example \ref{example: picard stack} using $q_1$ (resp. $q_2$). Let $S$ be a scheme over $\mathcal{B}$ and let $(\mathcal{L}, \mathcal{F},d)$ be an $S$-point of $\on{Pic}^0(X)\times B\mathbb{G}_m\times \mathbb{Z}$, where $d\in \mathbb{Z}$, $\mathcal{L}\in \on{Pic}_{gp}(X\times_{\mathcal{B}} S)/p_S^*\on{Pic}_{gp}(S)$ and $\mathcal{F}\in \on{Pic}_{gp}(S)$. Here $\on{Pic}_{gp}(Y)$ stands for the Picard group of $Y$, and $p_S: X\times_{\mathcal{B}} S\to S$ is the projection map.  The transition automorphism
    \[
    \rho_2\circ\rho_1^{-1}: \on{Pic}^0(X)\times B\mathbb{G}_m\times \mathbb{Z}\to \on{Pic}^0(X)\times B\mathbb{G}_m\times \mathbb{Z}
    \]
    maps $(\mathcal{L}, \mathcal{F},d)$ to
    \[
    \rho_2\circ\rho_1^{-1}(\mathcal{L}, \mathcal{F},d)=(\mathcal{L}(dq_1-dq_2), \mathcal{F}\otimes (\mathcal{L}(dq_1)|_{{\{q_1}\}\times S})^{-1}\otimes \mathcal{L}(dq_1)|_{{\{q_2}\}\times S}, d).
    \]
    Our goal is to show that under the identification 
    \[
    (\on{Pic}^0(X)\times B\mathbb{G}_m\times \mathbb{Z})^{\vee}\cong \on{Pic}^0(X)\times B\mathbb{G}_m\times \mathbb{Z}
    \]
    in Example \ref{example: picard stack}, the dual map 
    \[
    (\rho_2\circ\rho_1^{-1})^{\vee}: \on{Pic}^0(X)\times B\mathbb{G}_m\times \mathbb{Z}\to \on{Pic}^0(X)\times B\mathbb{G}_m\times \mathbb{Z}
    \] 
    is the inverse of $\rho_2\circ\rho_1^{-1}$. 
    This desired statement can be deduced from the observation that the map
    \[
    \mathbb{Z}\to \on{Pic}^0(X)
    \]
    \[
    d\mapsto \mathcal{O}_X(dq_1-dq_2)
    \]
    is dual to the map 
    \[
    \on{Pic}^0(X) \to B\mathbb{G}_m
    \]
    \[
    \mathcal{L}\mapsto\mathcal{L}|_{{\{q_1}\}\times S}\otimes (\mathcal{L}|_{{\{q_2}\}\times S})^{-1}.
    \]
\end{proof}

Now we consider a commutative group stack $\mathcal{P}$ over $\mathcal{B}$ together with a $\mu_n$-gerbe $\alpha: \mathfrak{P}\to \mathcal{P}$ on it. We assume this gerbe $\alpha$ is \'etale locally trivial over $\mathcal{B}$. We say $\alpha$ admits a group structure if it can be put into a short exact sequence of commutative group stacks as follows:
\begin{equation}\label{eq: gerbe exact sequence}
0\xrightarrow{} B\mu_n\xrightarrow{} \mathfrak{P}\xrightarrow{\alpha}\mathcal{P}\xrightarrow{} 0.
\end{equation}
We associate with such a $\mu_n$-gerbe $\alpha$ the following scheme $\widetilde{\on{Split}}^c(\mathcal{P}, \alpha)$ which is closely related to the principal component of relative splittings defined in (\ref{eq; definition of spitting}). 
\begin{definition}\label{def: splitting commutative}
(a) We define $\widetilde{\on{Split}}_{\mu_n}^c(\mathcal{P}, \alpha)$ to be the stack that classifies 1-morphisms of group stacks between $\mathfrak{P}$ and $B\mu_n$ that induce splittings of the short exact sequence (\ref{eq: gerbe exact sequence}). 

(b) We define $\widetilde{\on{Split}}^c(\mathcal{P}, \alpha)$ to be the stack that classifies 1-morphisms of group stacks between $\mathfrak{P}$ and $B\mathbb{G}_m$ that induce splittings of the short exact sequence
\begin{equation}\label{eq: G_m gerbe short exact sequence}
   0\xrightarrow{} B\mathbb{G}_m\xrightarrow{} \mathfrak{P}\times ^{B\mu_n}B\mathbb{G}_m\xrightarrow{}\mathcal{P}\xrightarrow{} 0. 
\end{equation}
This short exact sequence comes from the $\mathbb{G}_m$-gerbe induced from $\alpha$. 
\end{definition}

\begin{remark}\label{remark: commutative splitting}
(a) The stack $\widetilde{\on{Split}}_{\mu_n}^c(\mathcal{P}, \alpha)$ is a torsor for $\mathcal{P}^{\vee}[n]\coloneqq \on{Hom}(\mathcal{P}, B\mu_n)$ and we have an isomorphism $\widetilde{\on{Split}}^c(\mathcal{P}, \alpha)\cong\widetilde{\on{Split}}_{\mu_n}^c(\mathcal{P}, \alpha)\times^{\mathcal{P}^{\vee}[n]}\mathcal{P}^{\vee}$. 

(b) Consider the dual short exact sequence of (\ref{eq: G_m gerbe short exact sequence}):
\[
 0\xrightarrow{} \mathcal{P}^{\vee}\xrightarrow{} (\mathfrak{P}\times ^{B\mu_n}B\mathbb{G}_m)^{\vee}\xrightarrow{\sigma}\mathbb{Z}\xrightarrow{} 0.   
\]
It follows from the definition above that $\widetilde{\on{Split}}^c(\mathcal{P}, \alpha)$ is isomorphic to $\sigma^{-1}(1)$, which is a $\mathcal{P}^{\vee}$-torsor. 
\end{remark}

Now we assume $\mathcal{P}$ is an abelian scheme over $\mathcal{B}$. Let $\mathfrak{P}'$ be a $\mathfrak{P}$-torsor and let $\mathcal{P}'$ be its $\mu_n$-rigidification. The natural map $\alpha': \mathfrak{P}'\longrightarrow \mathcal{P}'$ gives $\mathfrak{P}'$ the structure of a $\mu_n$-gerbe over $\mathcal{P}'$. We associate with $\alpha'$ the principal component $\on{Split}'(\mathcal{P}', \alpha')$ of relative splittings defined in (\ref{eq; definition of spitting}), which is a $\mathcal{P}^{\vee}$-torsor. From $\mathfrak{P}'$ we can construct a commutative group stack $\mathfrak{P}_{\mathbb{Z}}$ with a group morphism $\sigma: \mathfrak{P}_{\mathbb{Z}}\to \mathbb{Z}$ such that $\sigma^{-1}(0)= \mathfrak{P}$, $\sigma^{-1}(1)= \mathfrak{P}'$ and each $\sigma^{-1}(m)$ is a $\mathfrak{P}$-torsor for any $m\in \mathbb{Z}$. The $\mu_n$-rigidification of $\mathfrak{P}_{\mathbb{Z}}$ is a commutative group scheme $\mathcal{P}_{\mathbb{Z}}$ that fits into a short exact sequence
\begin{equation}\label{eq: P-P_Z short exact sequence}
     0\xrightarrow{} \mathcal{P}\xrightarrow{} \mathcal{P}_{\mathbb{Z}}\xrightarrow{\sigma}\mathbb{Z}\xrightarrow{} 0.  
\end{equation}

The commutative group stack $\mathfrak{P}_{\mathbb{Z}}$ forms a $\mu_n$-gerbe over $\mathcal{P}_{\mathbb{Z}}$ which we denote by $\alpha_{\mathbb{Z}}$. Note that the stack of splittings $\widetilde{\on{Split}}_{\mu_n}^c(\mathcal{P}_{\mathbb{Z}}, \alpha_{\mathbb{Z}})$ is a $\mathcal{P}_{\mathbb{Z}}^{\vee}[n]$-torsor, and by the short exact sequence (\ref{eq: P-P_Z short exact sequence}), $\mathcal{P}_{\mathbb{Z}}^{\vee}[n]$ is \'etale locally isomorphic to $\mathcal{P}^{\vee}[n]\times B\mu_n$. It follows that $\widetilde{\on{Split}}_{\mu_n}^c(\mathcal{P}_{\mathbb{Z}}, \alpha_{\mathbb{Z}})$ forms a $\mu_n$-gerbe over its $\mu_n$-rigidification denoted by $\on{Split}_{\mu_n}^c(\mathcal{P}_{\mathbb{Z}}, \alpha_{\mathbb{Z}})$, and $\on{Split}_{\mu_n}^c(\mathcal{P}_{\mathbb{Z}}, \alpha_{\mathbb{Z}})$ is a $\mathcal{P}^{\vee}[n]$-torsor. Similarly $\widetilde{\on{Split}}^c(\mathcal{P}_{\mathbb{Z}}, \alpha_{\mathbb{Z}})$ forms a $\mathbb{G}_m$-gerbe over $\on{Split}^c(\mathcal{P}_{\mathbb{Z}}, \alpha_{\mathbb{Z}})$, and $\on{Split}^c(\mathcal{P}_{\mathbb{Z}}, \alpha_{\mathbb{Z}})$ is a $\mathcal{P}^{\vee}$-torsor. 

From a $\mu_n$-gerbe with a group structure $\alpha: \mathfrak{P}\to \mathcal{P}$ over an abelian scheme $\mathcal{P}$, we've constructed three $\mathcal{P}^{\vee}$-torsors: $\on{Split}'(\mathcal{P}', \alpha')$ in \ref{eq; definition of spitting},  $\widetilde{\on{Split}}^c(\mathcal{P}, \alpha)$ in Definition \ref{def: splitting commutative} (b), and $\on{Split}^c(\mathcal{P}_{\mathbb{Z}}, \alpha_{\mathbb{Z}})$ in the previous paragraph. The following Lemma shows that those three $\mathcal{P}^{\vee}$-torsors are isomorphic to each other. 
\begin{lemma}\label{lemma: compare splitting def}
    There are isomorphisms 
    \[
   \widetilde{\on{Split}}^c(\mathcal{P}, \alpha)\longleftarrow\on{Split}^c(\mathcal{P}_{\mathbb{Z}}, \alpha_{\mathbb{Z}})\longrightarrow \on{Split}'(\mathcal{P}', \alpha')
    \]
of $\mathcal{P}^{\vee}$-torsors.
\end{lemma}
\begin{proof}
It is enough to construct maps of $\mathcal{P}^{\vee}$-torsors from  $\on{Split}^c(\mathcal{P}_{\mathbb{Z}}, \alpha_{\mathbb{Z}})$ to $\widetilde{\on{Split}}^c(\mathcal{P}, \alpha)$ and $\on{Split}'(\mathcal{P}', \alpha')$. We first construct the map from $\on{Split}^c(\mathcal{P}_{\mathbb{Z}}, \alpha_{\mathbb{Z}})$ to $\widetilde{\on{Split}}^c(\mathcal{P}, \alpha)$. Note that by restricting splittings of the gerbe $\alpha_{\mathbb{Z}}$ to $\sigma^{-1}(0)=\mathcal{P}$, we get a map $\widetilde{\on{Split}}^c(\mathcal{P}_{\mathbb{Z}}, \alpha_{\mathbb{Z}})\to \widetilde{\on{Split}}^c(\mathcal{P}, \alpha)$. The desired morphism comes from the universal property of rigidification. Then we construct the map from $\on{Split}^c(\mathcal{P}_{\mathbb{Z}}, \alpha_{\mathbb{Z}})$ to $\on{Split}'(\mathcal{P}', \alpha')$. By restricting splittings of the gerbe $\alpha_{\mathbb{Z}}$ to $\sigma^{-1}(1)=\mathcal{P}'$, we get a map $\widetilde{\on{Split}}_{\mu_n}^c(\mathcal{P}_{\mathbb{Z}}, \alpha_{\mathbb{Z}})\to \widetilde{\on{Split}}_{\mu_n}(\mathcal{P}', \alpha')$. By universal property of rigidification, we get a map $\on{Split}_{\mu_n}^c(\mathcal{P}_{\mathbb{Z}}, \alpha_{\mathbb{Z}})\to \on{Split}_{\mu_n}(\mathcal{P}', \alpha')$, which leads to the desired morphism. 
\end{proof}

Now we apply the discussion above to the specific case of $\on{SL}_n$- and $\on{PGL}_n$- Higgs bundles described in Section \ref{Subsection: sln vs pgln}. Consider the stack version of the norm map
\[
    \mathfrak{Nm}: \mathfrak{Pic}(\Sigma/\mathcal{A}^0)\longrightarrow \mathfrak{Pic}(X\times \mathcal{A}^0/\mathcal{A}^0).
\] 
Let $\check{\mathfrak{P}}=\mathfrak{Nm}^{-1}(\mathcal{O}_X)$ and let $\check{\mathfrak{P}}_L=\mathfrak{Nm}^{-1}(L)$ for a line bundle $L$ on $X$. Let $\hat{\mathfrak{P}}=\check{\mathfrak{P}}/\Gamma$ and $\hat{\mathfrak{P}}_e=\check{\mathfrak{P}}_{\mathcal{O}_X(eq)}/\Gamma$ where $\Gamma=\on{Pic}^0(X)[n]$. The action of $\on{Pic}(X)$ on $\mathfrak{Pic}(\Sigma/\mathcal{A}^0)$ comes from the isomorphism 
\begin{equation}\label{eq: splitting of picard stack}
     \mathfrak{Pic}(X)\cong \on{Pic}^0(X)\times B\mathbb{G}_m\times \mathbb{Z}
\end{equation}
which depends on the marked point $q\in X(\mathcal{B})$, see Example \ref{example: picard stack}. Note that $\check{\mathfrak{P}}_L$ is a $\check{\mathfrak{P}}$-torsor and $\hat{\mathfrak{P}}_e$ is a $\hat{\mathfrak{P}}$-torsor. The $\mu_n$-gerbes $\check{\alpha}$ and $\hat{\alpha}$ that appear in topological mirror symmetry are given by the natural maps
\[
    {\check{\mathfrak{P}}}_L\xrightarrow{\check{\alpha}}\on{\check{\mathcal{P}}}_L \quad \on{and} \quad {\hat{\mathfrak{P}}}_e\xrightarrow{\hat{\alpha}}\on{\hat{\mathcal{P}}}_e.
\]
Let ${\check{\mathfrak{P}}}_{\mathbb{Z}}\xrightarrow{\check{\sigma}}\mathbb{Z}$ be the $\mathbb{Z}$-family of $\check{\mathfrak{P}}$-torsors constructed from ${\check{\mathfrak{P}}}_{\mathcal{O}_X(dq)}$ such that ${\check{\sigma}}^{-1}(0)=\check{\mathfrak{P}}$ and ${\check{\sigma}}^{-1}(1)={\check{\mathfrak{P}}}_{\mathcal{O}_X(dq)}$. Let ${\hat{\mathfrak{P}}}_{\mathbb{Z}}\xrightarrow{\hat{\sigma}}\mathbb{Z}$ be the $\mathbb{Z}$-family of $\hat{\mathfrak{P}}$-torsors constructed from ${\hat{\mathfrak{P}}}_e$ such that ${\hat{\sigma}}^{-1}(0)=\hat{\mathfrak{P}}$ and ${\hat{\sigma}}^{-1}(1)={\hat{\mathfrak{P}}}_e$. By taking $\mu_n$-rigidification, we get two $\mu_n$-gerbes $\check{\alpha}_{\mathbb{Z}}: {\check{\mathfrak{P}}}_{\mathbb{Z}}\to \check{\mathcal{P}}_{\mathbb{Z}}$ and $\hat{\alpha}_{\mathbb{Z}}: {\hat{\mathfrak{P}}}_{\mathbb{Z}}\to \hat{\mathcal{P}}_{\mathbb{Z}}$. The following theorem describes the duality relation between the two $\mu_n$-gerbes $\check{\alpha}$ and $\hat{\alpha}$ on the moduli spaces of parabolic $\on{SL}_n$- and $\on{PGL}_n$-Higgs bundles. 
\begin{theorem}\label{thm: gerbe duality stronger version}
With the notations as above, we have the following isomorphisms
\[\widetilde{\on{Split}}^c(\check{\mathcal{P}}_{\mathbb{Z}}/\mathcal{A}^0, \check{\alpha}_{\mathbb{Z}}^e)\cong \hat{\mathfrak{P}}_{e}\times^{B\mu_n, [d]}B\mathbb{G}_m,\quad \widetilde{\on{Split}}^c(\hat{\mathcal{P}}_{\mathbb{Z}}/\mathcal{A}^0, \hat{\alpha}_{\mathbb{Z}}^d)\cong \check{\mathfrak{P}}_{{\mathcal{O}_X(dq)}}\times^{B\mu_n, [e]}B\mathbb{G}_m\]
of $(\check{\mathcal{P}}_{\mathbb{Z}})^{\vee}$ (resp. $(\hat{\mathcal{P}}_{\mathbb{Z}})^{\vee}$)-torsors. Here $[d]: B\mu_n\to B\mathbb{G}_m$ is defined by taking a line bundle to its $d$-th power. 
\end{theorem}
\begin{proof}
We start with the first isomorphism. The $\mathbb{G}_m$-gerbe induced from $\check{\alpha}_{\mathbb{Z}}$ lies in the following short exact sequence of commutative group stacks:
\[
0\xrightarrow{}B\mathbb{G}_m\xrightarrow{}\check{\mathfrak{P}}_{\mathbb{Z}}\times^{B\mu_n} B\mathbb{G}_m\xrightarrow{}\check{\mathcal{P}}_{\mathbb{Z}}\xrightarrow{}0.
\]
The dual short exact sequence is given by 
\[
0\xrightarrow{}\check{\mathcal{P}}_{\mathbb{Z}}^{\vee}\xrightarrow{}(\check{\mathfrak{P}}_{\mathbb{Z}}\times^{B\mu_n} B\mathbb{G}_m)^{\vee}\xrightarrow{\sigma}\mathbb{Z}\xrightarrow{}0.
\]
It follows from Remark \ref{remark: commutative splitting} that $\widetilde{\on{Split}}^c(\check{\mathcal{P}}_{\mathbb{Z}}/\mathcal{A}^0, \check{\alpha}_{\mathbb{Z}}^e)\cong\sigma^{-1}(e)$. Note that the commutative group stack $\check{\mathfrak{P}}_{\mathbb{Z}}\times^{B\mu_n} B\mathbb{G}_m$ lies in the short exact sequence
\begin{equation}\label{eq: defining ses for check p_Z}
   0\xrightarrow{}\check{\mathfrak{P}}_{\mathbb{Z}}\times^{B\mu_n} B\mathbb{G}_m\xrightarrow{}\mathfrak{Pic}^{d\mathbb{Z}}(\Sigma/\mathcal{A}^0)\xrightarrow{\underline{\mathfrak{Nm}}}\on{Pic}^0(X\times\mathcal{A}^0/\mathcal{A}^0)\xrightarrow{}0,  
\end{equation}
where $\mathfrak{Pic}^{d\mathbb{Z}}(\Sigma/\mathcal{A}^0)\coloneqq \displaystyle\coprod_{m\in \mathbb{Z}}\mathfrak{Pic}^{dm}(\Sigma/\mathcal{A}^0)$, and $\underline{\mathfrak{Nm}}$ is the composite of ${\mathfrak{Nm}}$ with the projection $\mathfrak{Pic}(X)\to \on{Pic}^0(X)$ in the isomorphism (\ref{eq: splitting of picard stack}). Note that $\mathfrak{Pic}^{d\mathbb{Z}}(\Sigma/\mathcal{A}^0)$ lies in the following short exact sequence
\[
0\xrightarrow{}\mathfrak{Pic}^{d\mathbb{Z}}(\Sigma/\mathcal{A}^0)\xrightarrow{}\mathfrak{Pic}(\Sigma/\mathcal{A}^0)\xrightarrow{\on{deg}}\mathbb{Z}/d\mathbb{Z}\to 0,  
\]
and the dual short exact sequence is given by 
\begin{equation}\label{eq: degree dZ line bundles}
 0\xrightarrow{}B\mu_d\xrightarrow{}\mathfrak{Pic}(\Sigma/\mathcal{A}^0)\xrightarrow{}\mathfrak{Pic}^{d\mathbb{Z}}(\Sigma/\mathcal{A}^0)^{\vee}\to 0   
\end{equation}
as the commutative group stack $\mathfrak{Pic}(\Sigma/\mathcal{A}^0)$ is self dual. We also have the short exact sequence
\[
0\to B{\mu_d}\to B\mathbb{G}_m\xrightarrow[]{[d]} B\mathbb{G}_m\to 0.
\]
It follows that 
\begin{equation}\label{eq: description of dual of Pic(dZ)}
    \mathfrak{Pic}^{d\mathbb{Z}}(\Sigma/\mathcal{A}^0)^{\vee}\cong \mathfrak{Pic}(\Sigma/\mathcal{A}^0)\times ^{B\mathbb{G}_m, [d]}B\mathbb{G}_m.
\end{equation}
Taking the dual short exact sequence of (\ref{eq: defining ses for check p_Z}), we get 
\[ 0\xrightarrow{}\on{Pic}^0(X\times\mathcal{A}^0/\mathcal{A}^0)\xrightarrow{}\mathfrak{Pic}^{d\mathbb{Z}}(\Sigma/\mathcal{A}^0)^{\vee}\xrightarrow{}(\check{\mathfrak{P}}_{\mathbb{Z}}\times^{B\mu_n} B\mathbb{G}_m)^{\vee}\xrightarrow{}0.
\]
Here the morphism $\on{Pic}^0(X\times\mathcal{A}^0/\mathcal{A}^0)\xrightarrow{}\mathfrak{Pic}^{d\mathbb{Z}}(\Sigma/\mathcal{A}^0)^{\vee}$ is given by the composite
\[
\on{Pic}^0(X\times\mathcal{A}^0/\mathcal{A}^0)\xrightarrow[]{\iota}\mathfrak{Pic}(X\times\mathcal{A}^0/\mathcal{A}^0)\xrightarrow[]{p_X^*}\mathfrak{Pic}(\Sigma/\mathcal{A}^0)\xrightarrow{\pi}\mathfrak{Pic}^{d\mathbb{Z}}(\Sigma/\mathcal{A}^0)^{\vee},
\]
where $\iota$ is inclusion from (\ref{eq: splitting of picard stack}), $p_X^*$ is pull-back along the projection $p_X:\Sigma\to X$ and $\pi$ is the map in (\ref{eq: degree dZ line bundles}). Combining with (\ref{eq: description of dual of Pic(dZ)}), we get \[(\check{\mathfrak{P}}_{\mathbb{Z}}\times^{B\mu_n} B\mathbb{G}_m)^{\vee}\cong (\mathfrak{Pic}(\Sigma/\mathcal{A}^0)/\on{Pic}^0(X))\times ^{B\mathbb{G}_m, [d]}B\mathbb{G}_m,
\]
and the map $(\check{\mathfrak{P}}_{\mathbb{Z}}\times^{B\mu_n} B\mathbb{G}_m)^{\vee}\xrightarrow[]{\sigma}\mathbb{Z}$ is given by the degree map of $\mathfrak{Pic}(\Sigma/\mathcal{A}^0)$. It follows that we have the following isomorphisms of $(\check{\mathcal{P}}_{\mathbb{Z}})^{\vee}$-torsors
\[
  \widetilde{\on{Split}}^c(\check{\mathcal{P}}_{\mathbb{Z}}/\mathcal{A}^0, \check{\alpha}_{\mathbb{Z}}^e)\cong\sigma^{-1}(e)\cong (\mathfrak{Pic}^e(\Sigma/\mathcal{A}^0)/\on{Pic}^0(X))\times ^{B\mathbb{G}_m, [d]}B\mathbb{G}_m\cong \hat{\mathfrak{P}}_{e}\times^{B\mu_n, [d]}B\mathbb{G}_m. 
\]

Now we turn to the second isomorphism. The $\mathbb{G}_m$-gerbe induced from $\hat{\alpha}_{\mathbb{Z}}$ lies in the following short exact sequence
\[
    0\xrightarrow{}B\mathbb{G}_m\xrightarrow{}\hat{\mathfrak{P}}_{\mathbb{Z}}\times^{B\mu_n}B\mathbb{G}_m\xrightarrow{}\hat{\mathcal{P}}_{\mathbb{Z}}\xrightarrow{}0.
\]
The dual short exact sequence is given by
\[
  0\xrightarrow{}\hat{\mathcal{P}}_{\mathbb{Z}}^{\vee}\xrightarrow{}(\hat{\mathfrak{P}}_{\mathbb{Z}}\times^{B\mu_n}B\mathbb{G}_m)^{\vee}\xrightarrow{\tau}\mathbb{Z}\xrightarrow{}0
\]
and we have  $\widetilde{\on{Split}}^c(\hat{\mathcal{P}}_{\mathbb{Z}}/\mathcal{A}^0, \hat{\alpha}_{\mathbb{Z}}^d)\cong\tau^{-1}(d)$. Note that $\hat{\mathfrak{P}}_{\mathbb{Z}}\times^{B\mu_n}B\mathbb{G}_m$ lies in the following short exact sequence
\[ 0\xrightarrow{}\on{Pic}^0(X\times\mathcal{A}^0/\mathcal{A}^0)\xrightarrow{}\mathfrak{Pic}^{e\mathbb{Z}}(\Sigma/\mathcal{A}^0)\xrightarrow{}\hat{\mathfrak{P}}_{\mathbb{Z}}\times^{B\mu_n}B\mathbb{G}_m\xrightarrow{}0,
\]
and the dual short exact sequence is given by 
\[
0\xrightarrow{}(\hat{\mathfrak{P}}_{\mathbb{Z}}\times^{B\mu_n}B\mathbb{G}_m)^{\vee}\xrightarrow{}\mathfrak{Pic}^{e\mathbb{Z}}(\Sigma/\mathcal{A}^0)^{\vee}\xrightarrow{\eta}\on{Pic}^0(X\times\mathcal{A}^0/\mathcal{A}^0)\xrightarrow{}0. 
\]
By (\ref{eq: description of dual of Pic(dZ)}), we have $\mathfrak{Pic}^{e\mathbb{Z}}(\Sigma/\mathcal{A}^0)^{\vee}\cong \mathfrak{Pic}(\Sigma/\mathcal{A}^0)\times^{B\mathbb{G}_m, [e]} B\mathbb{G}_m$. The map $\eta$ is the composite
\[
\mathfrak{Pic}(\Sigma/\mathcal{A}^0)\times^{B\mathbb{G}_m, [e]} B\mathbb{G}_m\to \on{Pic}(\Sigma/\mathcal{A}^0)\xrightarrow[]{\on{Nm}}\on{Pic}(X\times \mathcal{A}^0/\mathcal{A}^0)\xrightarrow[]{}\on{Pic}^0(X\times \mathcal{A}^0/\mathcal{A}^0),
\]
where the first morphism is the rigidification map and the third morphism maps $L$ to $L\otimes \mathcal{O}_X(-\on{deg}(L)q)$. It follows that we have isomorphisms of $(\hat{\mathcal{P}}_{\mathbb{Z}})^{\vee}$-torsors
\[
  \widetilde{\on{Split}}^c(\hat{\mathcal{P}}_{\mathbb{Z}}/\mathcal{A}^0, \hat{\alpha}_{\mathbb{Z}}^d)\cong\tau^{-1}(d)\cong \check{\mathfrak{P}}_{\mathcal{O}(dq)}\times^{B\mu_n, [e]}B\mathbb{G}_m. 
\]
\end{proof}

\begin{proof}[Proof of Proposition \ref{prop: dual Hitchin system}]
 We first note that the moduli stack $\mathcal{M}^{e, \boldsymbol{\alpha}_2}_{\on{PGL}_n}$ is admissible in the sense of Definition \ref{def: admissible stack}. This follows from the isomorphism $\mathcal{M}^{e, \boldsymbol{\alpha}_2}_{\on{PGL}_n}\cong [\mathcal{M}^{L_e, \boldsymbol{\alpha}_2}_{\on{SL}_n}/\Gamma]$ in (\ref{eq: sln vs pgln}) and the observation that $\Gamma$ acts freely on $(\mathcal{M}^{L_e, \boldsymbol{\alpha}_2}_{\on{SL}_n})^0$, which is a consequence of Proposition \ref{prop: picard injective}. Properness of the Hitchin morphisms $\check{h}$ and $\hat{h}$ follows from Theorem \ref{thm: Hitchin map proper}. Torsor structures over $\mathcal{A}^0$ and Definition \ref{def: dual Hitchin fibration} Part (a) follows from discussions in Section \ref{Subsection: sln vs pgln}. Part (b) follows from Theorem \ref{thm: gerbe duality stronger version} and Lemma \ref{lemma: compare splitting def}. For part (c), let $a\in \mathcal{A}(\mathcal{O}_F)\cap \mathcal{A}^0(F)$ with corresponding $a_F\in \mathcal{A}^0(F)$. We need to show that if both Hitchin fibers  $\check{h}^{-1}(a_F)=(\check{\mathcal{P}}_{\mathcal{O}(d'q)})_{a_F}$ and $\hat{h}^{-1}(a_F)=(\hat{\mathcal{P}}_{e'})_{a_F}$ have $F$-rational points, then both $\mathbb{G}_m$-gerbes induced from $\check{\alpha}^{e'}$ and $\hat{\alpha}^{d'}$ split over $a_F$. For the $\on{SL}_n$-side, note that since the Hitchin map $\check{h}$ is proper, an $F$-point of the Hitchin fiber $\hat{h}^{-1}(a_F)$ extends to an $\mathcal{O}_F$-point. The desired splitting of the $B\mathbb{G}_m$-gerbe induced from $\check{\alpha}^{e'}|_{\check{h}^{-1}(a_F)}$ follows from Lemma 6.5 in \cite{MWZ}. Now we turn to the $\on{PGL}_n$-side. Since the $B\mathbb{G}_m$-gerbe induced from $\check{\alpha}^{e'}|_{\check{h}^{-1}(a_F)}$ splits, an $F$-point in the Hitchin fiber $(\check{\mathcal{P}}_{\mathcal{O}(d'q)})_{a_F}$ lifts to an $F$-point in $(\check{\mathfrak{P}}_{\mathcal{O}(d'q)}\times ^{B\mu_n, [e']}B\mathbb{G}_m)_{a_F}$, which in turn leads to a splitting of the $B\mathbb{G}_m$-gerbe induced from $\hat{\alpha}^{d'}|_{\hat{h}^{-1}(a_F)}$ by Theorem \ref{thm: gerbe duality stronger version}. 
\end{proof}

\begin{remark}
When $m = (1, 1, . . . , 1)$, i.e. we consider complete flags at the marked point, the family of curves $\Sigma/\mathcal{A}^0$ admits a global section $\tilde{q}\in \Sigma(\mathcal{A}^0)$ that lies above the marked point $q$. The existence of this global section implies that $\check{\mathfrak{P}}_{\mathcal{O}(dq)}$ is isomorphic to $\check{\mathcal{P}}\times B\mu_n$. By Theorem \ref{thm: gerbe duality stronger version}, this further implies that the $B\mathbb{G}_m$-gerbe induced from $\hat{\alpha}^{d'}$ splits over $(\mathcal{M}^{e, \boldsymbol{\alpha}_2}_{\on{PGL}_n})^0=\mathcal{M}^{e, \boldsymbol{\alpha}_2}_{\on{PGL}_n}\times _{\mathcal{A}}\mathcal{A}^0$. It follows that the gerbe $\hat{\alpha}^{d'}$ does not affect the computation of twisted $E$-polynomial, and the equality of $E$-polynomials in Theorem \ref{thm: main theorem} becomes
\[
 E(\mathcal{M}^{L, \boldsymbol{\alpha}_1}_{\on{SL}_n}; u,v)=E_{st}(\mathcal{M}^{e, \boldsymbol{\alpha}_2}_{\on{PGL}_n}; u,v).
\]
This agrees with the observation in \cite{GO} that the gerbe on the $\on{PGL}_n$-side is not needed to formulate the topological mirror symmetry for parabolic Higgs bundles with complete flags. 
\end{remark}

\section{Algebraic Volume forms on moduli spaces of parabolic Higgs bundles}\label{Section: Volume forms}
\subsection{Existence of algebraic gauge forms} \label{subsection: existence of volume forms}
Let $X$ be a smooth projective curve over a field $k$ with a marked point $q\in X(k)$. Let $M$ be a line bundle on $X$. In this subsection, we denote by $\mathcal{M}_L\subseteq\mathcal{M}_{\on{SL}_n}^{L, \boldsymbol{\alpha}}$ the \emph{stable} locus in the moduli space of \emph{semistable} $M$-twisted parabolic $\on{SL}_n$-Higgs bundles with parabolic weights $\boldsymbol{\alpha}$ and determinant $L$. For generic parabolic weights $\alpha$, we have $\mathcal{M}_L=\mathcal{M}_{\on{SL}_n}^{L, \boldsymbol{\alpha}}$. Recall that $\Gamma=\on{Pic}(X)[n]$ acts on $\mathcal{M}_L$ via tensor product. 

\begin{prop}\label{prop: trializing section exists}
    The moduli space $\mathcal{M}_L$ admits a $\Gamma$-invariant gauge form. 
\end{prop}

The proof of Proposition \ref{prop: trializing section exists} will occupy the rest of this subsection. We denote by ${\mathfrak{M}}$ the moduli \emph{stack} of stable $M$-twisted parabolic $\on{SL}_n$-Higgs bundles with parabolic weights $\boldsymbol{\alpha}$ and determinant $L$. Since ${\mathfrak{M}}$ is a $\mu_n$-gerbe over $\mathcal{M}$, proving Proposition \ref{prop: trializing section exists} is equivalent to showing that $\Omega_{{\mathfrak{M}}}^{\on{top}}$ admits a $\Gamma$-invariant trivializing section. 

We fix a presentation $M=K(D_1-D_2)$, where $D_1$ and $D_2$ are effective Weil divisors disjoint from $q$. We define the following stacks which will play an important role in the proof of Proposition \ref{prop: trializing section exists}.

\begin{definition}
(a) We denote by $\mathfrak{M}_{D_1}$ the moduli stack of $K(D_1)$-twisted quasi-parabolic $\on{SL}_n$-Higgs bundles of determinant $L$.

(b) We denote by $\mathfrak{N}$ the moduli stack of quasi-parabolic vector bundles of determinant $L$ together with a $D_1$-level structure. More precisely, if we let $D_1=\sum a_ip_i$, the moduli stack $\mathfrak{N}$ classifies vector bundles on $X$ of determinant $L$ together with a partial flag structure at $q$ and a trivialization on the $a_i$-th formal neighborhood of $p_i$ for each $i$.
\end{definition}
\begin{lemma}
 Both $\mathfrak{M}_{D_1}$ and $\mathfrak{N}$ are algebraic stacks locally of finite type over $k$.    
\end{lemma}
\begin{proof}
Let $\on{Bun}_n$ be the moduli stack of vector bundles of rank $n$ on $X$. We consider the forgetful maps $\mathfrak{M}_{D_1}\to \on{Bun}_n$ and $\mathfrak{N}\to \on{Bun}_n$. We note that those forgetful maps are representable and locally of finite presentation. Since $\on{Bun}_n$ is an algebraic stack locally of finite type over $k$, both $\mathfrak{M}_{D_1}$ and $\mathfrak{N}$ are algebraic stacks locally of finite type over $k$. 
\end{proof}

\begin{remark}\label{rk: various moduli stacks for constructing gauge forms}
(a) The cotangent stack $T^{*}\mathfrak{N}$ classifies $K(D_1)$-twisted quasi-parabolic $\on{SL}_n$-Higgs bundles of determinant $L$ together with a $D_1$-level structure. By forgetting this $D_1$-level structure, we get a natural map $\pi_{D_1}: T^{*}\mathfrak{N}
\longrightarrow \mathfrak{M}_{D_1}$, which gives $T^{*}\mathfrak{N}$ the structure of a $\on{SL}_n(D_1)$-torsor over $\mathfrak{M}_{D_1}$. By Cohen's structure theorem, we get an isomorphism
\[
    \on{SL}_n(D_1)\cong \prod_i\on{SL}_n(k_i[x]/(x^{a_i})),
\]
where $k_i$ is the residue field of $p_i$. 

(b) There is a natural inclusion map $\iota: \mathfrak{M}\longrightarrow \mathfrak{M}_{D_1}$ induced by the inclusion
\[\Gamma(X, \mathcal{E}nd_0(E)(D_1-D_2))\subseteq \Gamma(X, \mathcal{E}nd_0(E)(D_1)),\]
of Higgs fields, where $\mathcal{E}nd_0(E)$ is the sheaf of trace zero endomorphisms of $E$. 

(c) Each stack $\mathfrak{M}_{D_1}$ and $T^{*}\mathfrak{N}$ admits a $\Gamma$-action that is defined similarly as the $\Gamma$-action on $\mathfrak{M}$. Both maps $\pi_{D_1}: T^{*}\mathfrak{N}
\longrightarrow \mathfrak{M}_{D_1}$ and $\iota: \mathfrak{M}\longrightarrow \mathfrak{M}_{D_1}$ are $\Gamma$-equivariant with respect to those $\Gamma$-actions. 
\end{remark}

Now we study the deformation of $T^{*}\mathfrak{N}$ at $\boldsymbol{E}=(E, \phi, E_q^{\bullet}, \varphi, \theta)$, where $(E, \phi, E_q^{\bullet}, \varphi)$ is a $K(D_1)$-twisted quasi-parabolic $\on{SL}_n$-Higgs bundle and $\theta$ is a $D_1$-level structure of $E$. We denote by $\mathcal{P}ar_0(E)$ the sheaf of trace zero endomorphisms of $E$ that map $E_q^i$ to $E_q^i$ at $q$, and by $\mathcal{SP}ar_0(E)$ the sheaf of trace zero endomorphisms that map $E_q^i$ to $E_q^{i-1}$ at $q$. The deformation of $\boldsymbol{E}$ in $T^{*}\mathfrak{N}$ is governed by the complex
\begin{equation}\label{eq: tangent complex}
    \mathscr{F}^{\bullet}_{\boldsymbol{E}}=[\mathcal{P}ar_0(E)\otimes \mathcal{O}_X(-D_1)\xrightarrow{[\_, \phi]} \mathcal{SP}ar_0(E)\otimes K(D_1+q)]
\end{equation}
sitting at degree $-1$ and degree $0$: the hypercohomology $\mathbb{H}^0( \mathscr{F}^{\bullet}_{\boldsymbol{E}})$ gives the deformation of $\boldsymbol{E}$, and the hypercohomology $\mathbb{H}^{-1}( \mathscr{F}^{\bullet}_{\boldsymbol{E}})$ gives the lie algebra of the automorphism group of $\boldsymbol{E}$. 
The killing form on $\mathfrak{sl}_n$ induces an isomorphism $\mathcal{P}ar_0(E)\cong (\mathcal{SP}ar_0(E)\otimes \mathcal{O}_X(q))^{*}$. Under this isomorphism, the dual complex $(\mathscr{F}^{\bullet}_{\boldsymbol{E}})^{\vee}$ of $\mathscr{F}^{\bullet}_{\boldsymbol{E}}$ is given by
\[
(\mathscr{F}^{\bullet}_{\boldsymbol{E}})^{\vee}\simeq[\mathcal{P}ar_0(E)\otimes K^{-1}(-D_1)\xrightarrow{-[\_, \phi]} \mathcal{SP}ar_0(E)\otimes \mathcal{O}_X(D_1+q)]
\]
sitting at degree $0$ and degree $1$. We identify $(\mathscr{F}^{\bullet}_{\boldsymbol{E}})^{\vee}[1]\otimes K$ with $\mathscr{F}^{\bullet}_{\boldsymbol{E}}$ using the following isomorphism
\bd
\xymatrix{
 \mathscr{F}^{\bullet}_{\boldsymbol{E}}=[\mathcal{P}ar_0(E)\otimes \mathcal{O}_X(-D_1)\ar[r]^{[\_, \phi]}\ar[d]^{-1} &[\mathcal{SP}ar_0(E)\otimes K(D_1+q)]\ar[d]^{\on{id}}\\
(\mathscr{F}^{\bullet}_{\boldsymbol{E}})^{\vee}[1]\otimes K\simeq[\mathcal{P}ar_0(E)\otimes \mathcal{O}_X(-D_1)\ar[r]^{-[\_, \phi]}& [\mathcal{SP}ar_0(E)\otimes K(D_1+q)].
}
\ed
By Serre duality, we have an isomorphism $\mathbb{H}^{-1}( \mathscr{F}^{\bullet}_{\boldsymbol{E}})\cong \mathbb{H}^1( \mathscr{F}^{\bullet}_{\boldsymbol{E}})^{*}$. We denote by $(T^{*}\mathfrak{N})'\subset T^{*}\mathfrak{N}$ the maximal open Deligne-Mumford substack. We note that $\mathbb{H}^{-1}( \mathscr{F}^{\bullet}_{\boldsymbol{E}})=0$ for any $\boldsymbol{E}\in (T^{*}\mathfrak{N})'$ since the automorphism group is finite. The isomorphism $\mathbb{H}^1( \mathscr{F}^{\bullet}_{\boldsymbol{E}})\cong \mathbb{H}^{-1}( \mathscr{F}^{\bullet}_{\boldsymbol{E}})^{*}=0$ implies $(T^{*}\mathfrak{N})'$ is smooth. Indeed, smoothness of $(T^{*}\mathfrak{N})'$ can be checked by proving the lifting property for small extensions of finite-generated Artinian local $k$-algebras, and the obstruction for the existence of such liftings lies in $\mathbb{H}^1( \mathscr{F}^{\bullet}_{\boldsymbol{E}})$. The other isomorphism $\mathbb{H}^0( \mathscr{F}^{\bullet}_{\boldsymbol{E}})\cong \mathbb{H}^0( \mathscr{F}^{\bullet}_{\boldsymbol{E}})^{*}$ given by Serre duality induces a symplectic form on $(T^{*}\mathfrak{N})'$. This symplectic form is invariant under the action of $\Gamma$ and $\on{SL}_n(D_1)$. We denote by $\widetilde{\omega}_{D_1}$ the top exterior product of this symplectic form, which gives a $\Gamma$-invariant trivializing section of $\Omega^{\on{top}}_{(T^{*}\mathfrak{N})'}$. 

The goal is to construct a $\Gamma$-invariant trivializing section of $\Omega^{\on{top}}_{\mathfrak{M}}$ from $\widetilde{\omega}_{D_1}$. Our construction relies on the following theorem of M. Rosenlicht.
\begin{theorem}[cf. \cite{R} Theorem 3]\label{thm: Rosenlicht}
Let $G$ be a connected algebraic group over $k$ with identity element $e$. Let $f: G\longrightarrow \mathbb{G}_m$ be a map of $k$-schemes such that $f(e)=1$. Then $f$ is a character. 
\end{theorem}

\begin{corollary}\label{cor: constant functions}
The only invertible regular functions on $\on{SL}_n(k[x]/(x^a))$ are constant functions.
\end{corollary}
\begin{proof}
We consider the short exact sequence 
\[
    0\longrightarrow U\longrightarrow \on{SL}_n(k[x]/(x^a))\xrightarrow{\pi} \on{SL}_n(k)\longrightarrow 0,
\]
where $\pi$ is induced by the quotient map $k[x]/(x^a)\longrightarrow k$. We note that the kernel $U$ is a unipotent group over $k$. Indeed, for any $u\in U$, $u-I$ is a matrix with entries in the ideal $xk[x]/(x^a)$, therefore $(u-I)^a=0$. 
It follows that there are no non-trivial characters on $U$. It follows from \ref{thm: Rosenlicht} that all invertible regular functions on $\on{SL}_n(k[x]/(x^a))$ come from regular functions on $\on{SL}_n(k)$. Since $\on{SL}_n(k)$ is semi-simple, there are no non-trivial characters on $\on{SL}_n(k)$. Again it follows from Theorem \ref{thm: Rosenlicht} that there are no non-constant invertible regular functions on $\on{SL}_n(k)$.
\end{proof}
\begin{proof}[Proof of Proposition \ref{prop: trializing section exists}]
\emph{Step 1.} We denote by $\mathfrak{M}_{D_1}'\subset \mathfrak{M}_{D_1}$ the maximal Deligne-Mumford substack. The first step is to construct a trivializing section of $\Omega^{\on{top}}_{\mathfrak{M}_{D_1}'}$ from the trivializing section $\widetilde{\omega}_{D_1}$ of $\Omega^{\on{top}}_{(T^{*}\mathfrak{N})'}$. Recall that $T^{*}\mathfrak{N}$ is a $\on{SL}_n(D_1)$-torsor over $\mathfrak{M}_{D_1}$. Let $\mathcal{T}_{T^{*}\mathfrak{N}/\mathfrak{M}_{D_1}}$ be the relative tangent sheaf. The $\on{SL}_n(D_1)$-action on $T^{*}\mathfrak{N}$ induces a map 
\[\rho: \mathfrak{sl}_n(D_1)\longrightarrow \mathcal{T}_{T^{*}\mathfrak{N}/\mathfrak{M}_{D_1}}.
\] 
We fix a non-zero vector $\tau\in\wedge^{\on{top}}\mathfrak{sl}_n(D_1)$ in the top exterior product of the Lie algebra $\mathfrak{sl}_n(D_1)$. It follows from Corollary \ref{cor: constant functions} that the polyvector field $\rho(\tau)\in \wedge^{\on{top}}\mathcal{T}_{T^{*}\mathfrak{N}/\mathfrak{M}_{D_1}}$ is $\on{SL}_n(D_1)$-invariant. It is also $\Gamma$-invariant since the $\Gamma$-action on $T^{*}\mathfrak{N}$ commutes with the $\on{SL}_n(D_1)$-action. Now we consider the contraction $\iota_{\rho(\tau)}\widetilde{\omega}_{D_1}$ of the top-degree form $\widetilde{\omega}_{D_1}$ with $\rho(\tau)$. By definition, $\iota_{\rho(\tau)}\widetilde{\omega}_{D_1}$ is characterized by the property that $<{\displaystyle \iota _{\rho(\tau)}\widetilde{\omega}_{D_1} ,Y>=<\widetilde{\omega}_{D_1}, \rho(\tau)\wedge Y}>$ for any polyvector field $Y$ of the correct degree. Recall that $T^{*}\mathfrak{N}$ is a $\on{SL}_n(D_1)$-torsor over $\mathfrak{M}_{D_1}$, see Remark \ref{rk: various moduli stacks for constructing gauge forms} (a). The differential form
$\iota_{\rho(\tau)}\widetilde{\omega}_{D_1}$ is both horizontal and $\on{SL}_n(D_1)$-invariant, therefore descends to a top-degree form $\omega_{D_1}$ on $\mathfrak{M}_{D_1}'$, which is a $\Gamma$-invariant trivializing section. 

\emph{Step 2.} In this step we construct a trivializing section of $\Omega^{\on{top}}_{\mathfrak{M}}$ from the top-degree form $\omega_{D_1}$ on $\mathfrak{M}_{D_1}'$. Recall there is a natural inclusion map $\iota: \mathfrak{M}\longrightarrow \mathfrak{M}_{D_1}$. The image lies in the maximal Deligne-Mumford substack $\mathfrak{M}_{D_1}'\subset \mathfrak{M}_{D_1}$ since $\mathfrak{M}$ classifies stable Higgs bundles. Let $\boldsymbol{E}=(E, \phi, E_q^{\bullet}, \varphi)$ be a point of $\mathfrak{M}$. The deformation of $\boldsymbol{E}$ in $\mathfrak{M}$ is governed by
\[
  \mathscr{G}^{\bullet}_{\boldsymbol{E}}=[\mathcal{P}ar_0(E)\xrightarrow{[\_, \phi]} \mathcal{SP}ar_0(E)\otimes K(D_1-D_2+q)]
\]
sitting at degree $-1$ and degree $0$. The deformation of $\boldsymbol{E}$ in $\mathfrak{M}_{D_1}$ is governed by
\[
   \mathscr{H}^{\bullet}_{\boldsymbol{E}}=[\mathcal{P}ar_0(E)\xrightarrow{[\_, \phi]} \mathcal{SP}ar_0(E)\otimes K(D_1+q)]
\]
sitting at degree $-1$ and degree $0$.
We note that when restricted to $\mathfrak{M}$, the hypercohomology for both complexes are zero at degree $-1$ and $1$. Indeed, $\mathbb{H}^{-1}(\mathscr{G}^{\bullet}_{\boldsymbol{E}})$ (resp. $\mathbb{H}^{-1}(\mathscr{H}^{\bullet}_{\boldsymbol{E}})$) gives the lie algebra of the automorphism group for $\boldsymbol{E}$ in $\mathfrak{M}$ (resp. $\mathfrak{M}_{D_1})$. We have $\mathbb{H}^{-1}(\mathscr{G}^{\bullet}_{\boldsymbol{E}})=\mathbb{H}^{-1}(\mathscr{H}^{\bullet}_{\boldsymbol{E}})=0$ since $\boldsymbol{E}$ lies in the maximal Deligne-Mumford locus. The fact that $\mathbb{H}^{1}(\mathscr{H}^{\bullet}_{\boldsymbol{E}})=0$ follows from $\mathbb{H}^{1}(\mathscr{F}^{\bullet}_{\boldsymbol{E}})=0$ for the complex $\mathscr{F}^{\bullet}_{\boldsymbol{E}}$ in (\ref{eq: tangent complex}) and the long exact sequence for hypercohomology. The fact that $\mathbb{H}^{1}(\mathscr{G}^{\bullet}_{\boldsymbol{E}})=0$ follows from computation of the dimension of $\mathfrak{M}$ and $\mathfrak{M}_{D_1}$. It follows from the long exact sequence for hypercohomology that we have an isomorphism
\begin{equation}\label{eq: top form higgs}
    \iota^{*}\Omega^{\on{top}}_{\mathfrak{M}'_{D_1}}\cong\Omega^{\on{top}}_{\mathfrak{M}}\otimes \wedge^{\on{top}} \on{End}_0(E_{D_2}),
\end{equation}
where $\on{End}_0(E_{D_2})$ is the locally free sheaf on $\mathfrak{M}$ for which the fiber at $\boldsymbol{E}$ consists of trace zero endomorphisms of $E_{D_2}$. 

Now we consider the $\on{SL}_n(D_2)$-torsor $\pi: \widetilde{\mathfrak{M}}\longrightarrow \mathfrak{M}$ defined by adding a $D_2$-level structure. More precisely, $\widetilde{\mathfrak{M}}$ classifies quintuples $\boldsymbol{E}=(E, \phi, E_q^{\bullet}, \varphi, \theta)$, where $(E, \phi, E_q^{\bullet}, \varphi)$ is an element of $\mathfrak{M}$ and $\theta$ is a $D_2$-level structure of $E$. The $D_2$-level structure induces a canonical trivialization 
\begin{equation}\label{eq: trivilization of EndD2}
    \pi^{*}\on{End}_0(E_{D_2})\cong \mathcal{O}_{\widetilde{\mathfrak{M}}}\otimes_k \mathfrak{sl}_n(D_2).
\end{equation}
We fix a non-zero vector $\sigma\in\wedge^{\on{top}}\mathfrak{sl}_n(D_2)$, which induces a trivialization of $\pi^{*}(\wedge^{\on{top}}\on{End}_0(E_{D_2}))$ through isomorphism (\ref{eq: trivilization of EndD2}). 
It follows from Corollary \ref{cor: constant functions} that this trivialization of $\pi^{*}(\wedge^{\on{top}}\on{End}_0(E_{D_2}))$ descends to a trivialization
\[
    \wedge^{\on{top}}\on{End}_0(E_{D_2})\cong \mathcal{O}_{\mathfrak{M}}.
\]
By isomorphism (\ref{eq: top form higgs}), the trivializing section $\omega_{D_1}$ of $\Omega^{\on{top}}_{\mathfrak{M}'_{D_1}}$ induces a trivializing section $\omega$ of $\Omega^{\on{top}}_{\mathfrak{M}}$ which is still $\Gamma$-invariant. 
\end{proof}

\subsection{Comparison between degrees}\label{subsection: comparison between degrees}
Recall that we denote by $\mathcal{M}_{L}$ the moduli space of \emph{stable} $M$-twisted parabolic $\on{SL}_n$-Higgs bundles with parabolic weights $\boldsymbol{\alpha}$ and determinant $L$. In Subsection \ref{subsection: existence of volume forms}, we have constructed an algebraic gauge form $\omega_{L}$ on $\mathcal{M}_{L}$. The goal of this subsection is to discuss the relation between this $\omega_L$ for different choices of the line bundle $L$. 

The moduli space $\mathcal{M}_{L}$ admits the Hitchin map $\check{h}: \mathcal{M}_{L}\longrightarrow \mathcal{A}$ such that when restricted to the open subscheme $\mathcal{A}^0\subset \mathcal{A}$, $\mathcal{M}^0_{L}\coloneqq \mathcal{M}_L\times _{\mathcal{A}}{\mathcal{A}^0}$ becomes a torsor for the relative Prym scheme $\check{\mathcal{P}}/\mathcal{A}^0$. By the discussion at the end of Subsection \ref{subsection: abstract hitchin systems}, if we fix a trivializing section $\omega_{\mathcal{A}}\in \Gamma(\mathcal{A}, \Omega^{\on{top}}_{\mathcal{A}})$, the top-degree form $\omega_L$ can be written uniquely as a wedge product when restricted to $\mathcal{M}_L^0$: 
\begin{equation}\label{eq: wedge product; comparison of degree}
      \omega_{L}=\check{h}^{*}\omega_{\mathcal{A}}\wedge \tilde{\omega}_L,  
\end{equation}
where $\tilde {\omega}_L\in \Gamma(\mathcal{M}_L^0, \Omega^{\on{top}}_{\mathcal{M}_L^0/\mathcal{A}^0})$ is a translation invariant trivializing section. We denote by $\tilde{\omega}'_L$ the corresponding section in $\Gamma(\check{\mathcal{P}}, \Omega^{\on{top}}_{\check{\mathcal{P}}/\mathcal{A}^0})$, see Lemma \ref{lemma: top form yoga}. Then we have the following proposition. 

\begin{prop}\label{prop: comparision between degrees}
    The section $\tilde{\omega}'_L\in\Gamma(\check{\mathcal{P}}, \Omega^{\on{top}}_{\check{\mathcal{P}}/\mathcal{A}^0})$ does not depend on the choice of the line bundle $L$. 
\end{prop}

It is enough to prove Proportion \ref{prop: comparision between degrees} when the base field $k$ is algebraically closed, therefore we assume $k=\bar{k}$. Recall we denote by $\mathfrak{N}$ the moduli stack of quasi-parabolic vector bundles of determinant $L$ together with a $D_1$-level structure. We will use the notation $\mathfrak{N}_L$ when we want to emphasize the choice of $L$. Let $p_{\mathfrak{N}}: T^{*}\mathfrak{N}\longrightarrow \mathfrak{N}$ be the projection map. The cotangent stack $T^{*}\mathfrak{N}$ admits a tautological $1$-form $\theta_{\mathfrak{N}}$
\[
    \theta_{\mathfrak{N}}=dp_{\mathfrak{N}}\circ \delta: T^{*}\mathfrak{N}\longrightarrow T^{*}(T^{*}\mathfrak{N})
\]
defined by the composition of the differential of $p_{\mathfrak{N}}$
\[
    dp_{\mathfrak{N}}: T^{*}\mathfrak{N}\times_{\mathfrak{N}} T^{*}\mathfrak{N}\longrightarrow T^{*}(T^{*}\mathfrak{N})
\]
with the diagonal map 
\[
    \delta: T^{*}\mathfrak{N}\longrightarrow T^{*}\mathfrak{N}\times_{\mathfrak{N}} T^{*}\mathfrak{N}.
\]
We have the following 
\begin{lemma}
  Over the maximal open Deligne-Mumford locus $(T^{*}\mathfrak{N})'\subset T^{*}\mathfrak{N}$, the 2-form $d\theta_{\mathfrak{N}}$ is equal to the canonical symplectic form defined using Serre duality of the complex in (\ref{eq: tangent complex}).   
\end{lemma}
\begin{proof}
    Let $\boldsymbol{E}$ be a point of $(T^{*}\mathfrak{N})'$. We denote $\mathscr{F}^{\bullet}_{\boldsymbol{E}}=[\mathcal{F}^{-1}\to \mathcal{F}^0]$, where $\mathcal{F}^{-1}=\mathcal{P}ar_0(E)\otimes \mathcal{O}_X(-D_1)$ and  $\mathcal{F}^0=\mathcal{SP}ar_0(E)\otimes K(D_1+q)]$. By the spectral sequence for hypercohomology, we have the following short exact sequence
    \begin{equation}\label{eq: SES for hypercohomology}
      0\to E^{0,0}_2\to  \mathbb{H}^{0}( \mathscr{F}^{\bullet}_{\boldsymbol{E}})\to E^{-1,1}_2 \to 0,
    \end{equation}
    where $E^{0,0}_2=\on{Coker}(H^0(\mathcal{F}^{-1})\to H^0(\mathcal{F}^{0}))$ and $E^{-1,1}_2=\on{Ker}(H^1(\mathcal{F}^{-1})\to H^1(\mathcal{F}^{0}))$. Let $\pi: T^{*}\mathfrak{N}\to \mathfrak{N}$ be the projection map, and let $\iota: T^{*}_{\pi(\boldsymbol{E})}\mathfrak{N}\to T^{*}\mathfrak{N}$ be the inclusion of the cotangent fiber at $\pi(\boldsymbol{E})$. Note that we have $H^1(\mathcal{F}^{-1})\cong T_{\pi(\boldsymbol{E})}\mathfrak{N}$ and $H^0(\mathcal{F}^{0})\cong T^{*}_{\pi(\boldsymbol{E})}\mathfrak{N}$. The map $H^{0}(\mathcal{F}^0)\to E^{0,0}_2\to \mathbb{H}^{0}( \mathscr{F}^{\bullet}_{\boldsymbol{E}})$ in (\ref{eq: SES for hypercohomology}) is given by the tangent map of $\iota$, and the map $\mathbb{H}^{0}( \mathscr{F}^{\bullet}_{\boldsymbol{E}})\to E^{-1,1}_2\to H^{1}(\mathcal{F}^{-1}) $ in (\ref{eq: SES for hypercohomology}) is given by the tangent map of $\pi$. Now the desired statement follows from the definition of the two symplectic forms. 
\end{proof}
Since $T^{*}\mathfrak{N}$ classifies $K(D_1)$-twisted quasi-parabolic $\on{SL}_n$-Higgs bundles together with a $D_1$-level structure, we have the Hitchin map
\[
    h: T^{*}\mathfrak{N}\longrightarrow \mathcal{A}_{D_1}=\bigoplus_{i=2}^n\Gamma(X, K(D_1+q)^i). 
\]
Let $\widetilde{X}$ be the total spectral curve inside $T^{*}X(D_1+q)\times \mathcal{A}_{D_1}$, and let $p_X:\widetilde{X}\longrightarrow X$ be the projection map. We introduce the following notations.
\begin{definition}
(a) We denote by $X^0$ be the open curve defined by
\[
X^0=X-\on{Supp}(D_1)-\on{Supp}(D_2)-q. 
\]
(b) We denote by $\mathcal{A}_{D_1}^0\subset \mathcal{A}_{D_1}$ be the locus where the spectral curves are smooth and connected above $X^0$. Let $(T^{*}\mathfrak{N})^0=T^{*}\mathfrak{N}\times_{\mathcal{A}_{D_1}}\mathcal{A}_{D_1}^0$ and $\widetilde{X}^{0}=p_X^{-1}(X^0)$.
\end{definition}

\begin{remark}
Let $(E, \phi, E^{\bullet}_q)$ be a $K(D_1)$-twisted quasi-parabolic Higgs bundle that is mapped to $a\in \mathcal{A}_{D_1}^0$ under the Hitchin map. Since the spectral curve $\widetilde{X}_a$ is smooth above $X^0$, by the BNR correspondence \cite{BNR}, the spectral sheaf corresponding to $(E, \phi)$ restricts to a line bundle on $\widetilde{X}_a^0$. 
\end{remark}

Now we fix a closed point $p\in X^0(k)$. Let $\widetilde{X}_{(p)}^0$ be the restriction of $\widetilde{X}^0$ along $\mathcal{A}_{D_1}^0\xrightarrow{p\times \on{id}}X^{0}\times\mathcal{A}_{D_1}^0$. The Abel-Jacobi map
\[
    \widetilde{X}^0\times \on{Pic}(\widetilde{X}^0)\longrightarrow \on{Pic}(\widetilde{X}^0)
\]
induces a map
\[
    a_p: \widetilde{X}_{(p)}^0\times_{\mathcal{A}_{D_1}^0}  (T^{*}\mathfrak{N}_{L})^0\longrightarrow (T^{*}\mathfrak{N}_{L(p)})^0. 
\]
by modifying the spectral sheaf above $X^0$. The following proposition is an analogue of Theorem 4.12 in \cite{BB}. 
\begin{prop}\label{prop: BB}
    $a_p^* \theta_{\mathfrak{N}_{L(p)}}=p_2^{*}\theta_{\mathfrak{N}}$, where $p_2: \widetilde{X}_{(p)}^0\times_{\mathcal{A}_{D_1}^0}  (T^{*}\mathfrak{N}_{L})^0\longrightarrow (T^{*}\mathfrak{N}_{L})^0$ is the projection to the second factor. 
\end{prop}
\begin{proof}
The proof is essentially the same as in \cite{BB}. We consider the  moduli stack $\mathcal{H}ecke^1_p$ of triples \[(\boldsymbol{E}, \boldsymbol{F}, i)\]
where $\boldsymbol{E}\in \mathfrak{N}_L$, $\boldsymbol{F}\in \mathfrak{N}_{L(p)}$ and $i: E\hookrightarrow F$ is an inclusion of the underlying vector bundles such that $F/E$ is the simple skyscraper sheaf at $p\in X^0(k)$ and the partial flag structures and $D_1$-level structures on $\boldsymbol{E}$ and $\boldsymbol{F}$ coincide under $i$. We consider the following maps:
\bd
\xymatrix{
& \mathcal{H}ecke^1_p\ar[dl]_{q} \ar[dr]^{p}\\
\mathfrak{N}_{L(p)} & &\mathfrak{N}_{L}
}
\ed
where $q$ maps the triple to $\boldsymbol{F}$ and $p$ maps the triple to $\boldsymbol{E}$. Both $p$ and $q$ are smooth.

Consider the following pull-back diagram:
\bd
\xymatrix{
Z^0\ar[rr]^{f_1} \ar[d]^{f_2} &&  q^{*}(T^{*}\mathfrak{N}_{L(p)})^0\ar[d]^{dq}\\
p^{*}(T^{*}\mathfrak{N}_{L})^0\ar[rr]^{dp} && T^{*}\mathcal{H}ecke^1_p.
}
\ed
The stack $Z^0$ classifies the following data $((\boldsymbol{E}, \phi_E),(\boldsymbol{F}, \phi_F), i: E\hookrightarrow F)$, where $(\boldsymbol{E}, \phi_E)\in (T^{*}\mathfrak{N}_{L})^0$, $(\boldsymbol{F}, \phi_F)\in (T^{*}\mathfrak{N}_{L(p)})^0$ and $(\boldsymbol{E}, \boldsymbol{F}, i)\in \mathcal{H}ecke_p^1$. Since the twisted Higgs bundles $(E, \phi_E)$ and $(F, \phi_F)$ are isomorphic away from $p$, they map to the same point $a\in \mathcal{A}^0_{D_1}$ under the Hitchin map and the corresponding spectral sheaves on ${\Sigma}_a$ differ by a simple skyscraper sheaf at some $p'\in {\Sigma}_a$ that maps to $p$ under the projection map $p_X: {\Sigma}_a\longrightarrow X$. Recall that $\Sigma_a$ is the normalization of the spectral curve $\widetilde{X}_a$ and they are isomorphic above $X^0$. It follows that $Z^0$ is isomorphic to $ \widetilde{X}_{(p)}^0\times_{\mathcal{A}_{D_1}^0}(T^{*}\mathfrak{N}_L)^0$. The Abel-Jacobi map $a_p$ is identified with
\[
    \on{pr}_2\circ f_1: Z^0\longrightarrow (T^{*}\mathfrak{N}_{L(p)})^0,
\]
where $\on{pr}_2$ is the projection \[q^{*}(T^{*}\mathfrak{N}_{L(p)})^0=\mathcal{H}ecke^1_p\times_{q,\mathfrak{N}_{L(p)}}(T^{*}\mathfrak{N}_{L(p)})^0\longrightarrow (T^{*}\mathfrak{N}_{L(p)})^0.
\]
Similarly $p_2$ is identified with $\on{pr}_2\circ f_2: Z^0\longrightarrow (T^{*}\mathfrak{N}_{L})^0.$ Our goal is to show $a_p^{*}\theta_{\mathfrak{N}_{L(p)}}=p_2^{*}\theta_{\mathfrak{N}_L}$. Both 1-forms are equal to the pull-back of the tautological 1-form on $T^*\mathcal{H}ecke^1_p$ to $Z^0$.
\end{proof}

\begin{proof}[Proof of Proposition \ref{prop: comparision between degrees}]
Recall that $\tilde{\omega}_{D_1}$ is the trivializing section of $\Omega_{(T^{*}\mathfrak{N})'}^{\on{top}}$ obtained by taking the top wedge product of the canonical symplectic form, i.e. we have $\tilde{\omega}_{D_1}=\wedge^{\on{top}}d\theta_{\mathfrak{N}}$. Therefore Proposition \ref{prop: BB} implies \[a_p^{*}\tilde{\omega}_{D_1, L(p)}=p_2^{*}\tilde{\omega}_{D_1, L}.
\]
The stack $T^{*}\mathfrak{N}$ is a $\on{SL}_n(D_1)$-torsor over the moduli stack $\mathfrak{M}_{D_1}$ of $K(D_1)$-twisted quasi-parabolic $\on{SL}_n$-Higgs bundles. Let $\mathfrak{M}_{D_1}^0\subset \mathfrak{M}_{D_1}$ be the quotient of $(T^{*}\mathfrak{N})^0$ by this $\on{SL}_n(D_1)$-action. The gauge form $\omega_{D_1}$ on $\mathfrak{M}^0_{D_1}$ is constructed by contracting $\widetilde{\omega}_{D_1}$ with $\rho(\tau)\in \wedge^{\on{top}}\mathcal{T}_{T^{*}\mathfrak{N}/\mathfrak{M}_{D_1}}$, where $\rho$ is the Lie algebra action
\[\rho: \mathfrak{sl}_n(D_1)\longrightarrow \mathcal{T}_{T^{*}\mathfrak{N}/\mathfrak{M}_{D_1}}
\] 
and $\tau\in\wedge^{\on{top}}\mathfrak{sl}_n(D_1)$ is a non-zero vector. Since $a_p$ is $\on{SL}_n$-equivariant, we have $(a_p)_{*}(\rho(\tau))=\rho(\tau)$. The map $a_p$ descends to 
\[
    \bar{a}_p: \widetilde{X}_{(p)}^0\times_{\mathcal{A}_{D_1}^0}  \mathfrak{M}_{D_1,L}^0\longrightarrow \mathfrak{M}_{D_1,L(p)}^0,
\]
and we have 
\[\
\bar{a}_p^{*}\omega_{D_1,L(p)}=p_2^*\omega_{D_1, L}.
\]

The restriction of $\bar{a}_p$ induces a map
\[
    \bar{a}_p: \widetilde{X}_{(p)}^0\times_{\mathcal{A}^0}  \mathcal{M}_{L}^0\longrightarrow \mathcal{M}_{L(p)}^0,
\]
which we still call $\bar{a}_p$. The gauge form $\omega_L$ (resp. $\omega_{L(p)}$) on $\mathcal{M}^0_L$ (resp. $\mathcal{M}^0_{L(p)}$) is constructed by isomorphism (\ref{eq: top form higgs}) and the trivialization of $\wedge^{\on{top}} \on{End}_0(E_{D_2})$. Since the map $\bar{a}_p$ is defined by modifying the Higgs bundle $(E, \phi)$ at $p\in X^0(k)$, it doesn't affect the fiber $E_{D_2}$ of $E$ at $D_2$. Therefore we have 
\[
    \bar{a}_p^{*}\omega_{L(p)}=p_2^*\omega_{L}.
\]
Combining equation (\ref{eq: wedge product; comparison of degree}), we have
\[
    \bar{a}_p^*(\check{h}^*\omega_{\mathcal{A}}\wedge \widetilde{\omega}_{L(p)})= p_2^*(\check{h}^*\omega_{\mathcal{A}}\wedge \widetilde{\omega}_{L}).
\]
Let $a\in\mathcal{A}^0(k)$. We denote by $(\mathcal{M}^0_{L})_a$ (resp. $(\mathcal{M}^0_{L(p)})_a$) the Hitchin fiber above $a$, which is a $\check{P}_a$-torsor. Let $\tilde{p}\in \widetilde{X}_{(p)}^0(k)$ be a closed point in the spectral curve over $a$. The restriction of $\bar{a}_p$ to $\tilde{p}$ induces an isomorphism of $\check{P}_a$-torsors
\[
    \bar{a}_{\tilde{p}}: (\mathcal{M}^0_{L})_a \longrightarrow (\mathcal{M}^0_{L(p)})_a
\]
that satisfies 
\begin{equation}\label{eq: fiber wise gauge form}
    \bar{a}_{\tilde{p}}^{*}\widetilde{\omega}_{L(p)}=\widetilde{\omega}_{L}. 
\end{equation}
Since $\tilde{\omega}'_L\in\Gamma(\check{\mathcal{P}}, \Omega^{\on{top}}_{\check{\mathcal{P}}/\mathcal{A}^0})$ (resp. $\tilde{\omega}'_{L(p)}$) is obtained from $\widetilde{\omega}_{L}$ (resp. $\widetilde{\omega}_{L(p)}$) by pulling-back along local trivializations of the $\check{P}$-torsor $\mathcal{M}^0_L$ (resp. $\mathcal{M}^0_{L(p)})$, (\ref{eq: fiber wise gauge form}) implies $\widetilde{\omega}_{L}'=\widetilde{\omega}_{L(p)}'$ as desired. 
\end{proof}

The conclusions in Proposition \ref{prop: trializing section exists} and \ref{prop: comparision between degrees} also hold for the stable locus of the moduli space $\mathcal{M}_{\on{GL}_n}^{d, \boldsymbol{\alpha}}$ of semistable $M$-twisted parabolic $\on{GL}_n$-Higgs bundles of degree $d$. Recall that there exists and open subscheme $(\mathcal{M}_{\on{GL}_n}^{d, \boldsymbol{\alpha}})^0\subset \mathcal{M}_{\on{GL}_n}^{d, \boldsymbol{\alpha}}$ that is a $\mathcal{P}$-torsor, where $\mathcal{P}=\on{Pic}^0(\Sigma/\mathcal{A}_P^0)$. We have the following

\begin{corollary}\label{cor: gauge form GLn}
Assume that $char(k)$ and $n$ are coprime. For each integer $d$, there exists a gauge form $\omega_d$ on the stable locus of the moduli space $\mathcal{M}_{\on{GL}_n}^{d, \boldsymbol{\alpha}}$, such that if we fix a gauge form on $\mathcal{A}_P$, the corresponding $\tilde{\omega}_d'\in \Gamma({\mathcal{P}},\Omega^{\on{top}}_{{\mathcal{P}}/\mathcal{A}_P^0})$ is independent of $d$. 
\end{corollary}
\begin{proof}
Let $L$ be a line bundle on $X$ at degree $d$. The moduli spaces $\mathcal{M}_{\on{GL}_n}^{d, \boldsymbol{\alpha}}$ and $\mathcal{M}_{\on{SL}_n}^{L, \boldsymbol{\alpha}}$ are related by the following map:
\[
\Phi: \mathcal{M}_{\on{SL}_n}^{L, \boldsymbol{\alpha}}\times H^0(X, M)\times \on{Pic}^0(X)\longrightarrow \mathcal{M}_{\on{GL}_n}^{d, \boldsymbol{\alpha}}
\]
\[
((E,\phi,  E_q^{\bullet}, \varphi), s, F)\mapsto (E\otimes F,  \phi\otimes \on{id}+\on{id}\otimes s, E_q^{\bullet}\otimes F_q,). 
\]
The scheme on the left-hand side admits a free $\Gamma$-action defined by 
\[\gamma\cdot(E, s, F)=(E\otimes \gamma^{-1}, s, F\otimes \gamma),
\] 
and $\Phi$ can be identified with the quotient map of this $\Gamma$-action. In Proposition \ref{prop: trializing section exists}, we've constructed a $\Gamma$-invariant gauge form $\omega_L$ on $\mathcal{M}_{\on{SL}_n}^{L, \boldsymbol{\alpha}}$. We fix a gauge form $\omega_2$ on $H^0(X, M)$ and a translation invariant gauge form $\omega_3$ on $\on{Pic}^0(X)$. The gauge form $p_1^{*}\omega_L\wedge p_2^*\omega_2\wedge p_3^*\omega_3$ is $\Gamma$-invariant, therefore descends to a gauge form $\omega_d$ on $\mathcal{M}_{\on{GL}_n}^{d, \boldsymbol{\alpha}}$. We still need to show that when fixing a gauge form on $\mathcal{A}_P$, the corresponding $\tilde{\omega}_d'\in \Gamma({\mathcal{P}},\Omega^{\on{top}}_{{\mathcal{P}}/\mathcal{A}_P^0})$ is independent of $d$. The map $\Phi$ induces an isomorphism
\[
    \Phi_{\mathcal{A}}: \mathcal{A}\times H^{0}(X,M)\longrightarrow \mathcal{A}_P
\]
of the Hitchin base. A gauge form on $\mathcal{A}_P$ together with $\omega_2$ on $H^{0}(X,M)$ induces a gauge form on $\mathcal{A}$ through $\Phi_{\mathcal{A}}$. Now we discuss the relation between the corresponding $\tilde{\omega}_L'\in \Gamma(\check{\mathcal{P}},\Omega^{\on{top}}_{\check{\mathcal{P}}/\mathcal{A}^0})$ and $\tilde{\omega}_d'\in \Gamma({\mathcal{P}},\Omega^{\on{top}}_{{\mathcal{P}}/\mathcal{A}_P^0})$. Let $a\in \mathcal{A}_P(\bar{k})$ be a geometric point and let $\Phi_{\mathcal{A}}^{-1}(a)=(a', s)$, where $a'\in\mathcal{A}(\bar{k})$ and $s\in H^{0}(X_{\bar{k}}, M)$. Note that there is a canonical isomorphism $\Sigma_{a'}\cong \Sigma_a$ such that the morphism
\[
\Phi_{(a', s)}: h_{\on{SL}_n}^{-1}(a')\times \{s\}\times \on{Pic}^0(X)\longrightarrow h_{\on{GL}_n}^{-1}(a). 
\]
is identified with 
\[
\on{Prym}_{L}(\Sigma_{a'})\times \on{Pic}^0(X)\to \on{Pic}^d(\Sigma_{a})
\]
\[
(N, F)\mapsto N\otimes p_X^{*}F.
\]
It follows that the gauge form $p_1^*(\tilde{\omega}_L')_{a'}\wedge p_2^*\omega_3$ on $\on{Prym}(\Sigma_{a'})\times \on{Pic}^0(X)$ descends to $(\tilde{\omega}_d')_{a}$ on ${\mathcal{P}}_{a}=\on{Pic}^0(\Sigma_a)$ through the quotient map
\[
\on{Prym}(\Sigma_{a'})\times \on{Pic}^0(X)\longrightarrow (\on{Prym}(\Sigma_{a'})\times \on{Pic}^0(X))/\Gamma \cong \on{Pic}^0(\Sigma_a).
\] 
Since $(\tilde{\omega}_L')$ is independent of $L$ by Proposition \ref{prop: comparision between degrees}, $(\tilde{\omega}_d')$ is independent of $d$. 
\end{proof}

\section{Topological mirror symmetry for parabolic Higgs bundles}\label{section: proof of mirror symmetry}
\subsection{Proof of topological mirror symmetry}
The goal of this subsection is to prove Theorem \ref{thm: main theorem}.  We would like to apply $p$-adic integration to the dual pair of Hitchin systems 
\[
    (\mathcal{M}^{L_d, \boldsymbol{\alpha}}_{\on{SL}_n}(X), \mathcal{M}^{e, \boldsymbol{\alpha}'}_{\on{PGL}_n}(X),\mathcal{A}, \check{\alpha}^{e'}, \hat{\alpha}^{d'})
\]
in Proposition \ref{prop: dual Hitchin system} given by the moduli spaces of semistable parabolic $\on{SL}_n$- and $\on{PGL}_n$-Higgs bundles on a smooth complex projective curve $X$. The first step is to construct a dual pair of Hitchin systems over a finitely generated $\mathbb{Z}$-subalgebra $R\subset \mathbb{C}$ which acts as an $R$-model for the Hitchin systems over $\mathbb{C}$. Note that this dual pair of Hitchin systems depends on the data $(X, M, q)$, where $X$ is a smooth complex projective curve, $M$ is a line bundle on $X$, and $q\in X_\mathbb{C}(\mathbb{C})$ is a marked point. Let $X_R$ be an $R$-model of $X$ for a finitely generated $\mathbb{Z}$-subalgebra $R\subset \mathbb{C}$. We assume $n$ is invertible in $R$ and $R$ contains all the $n^{2g}$-th roots of unity. By possibly extending $R$, we assume $X_R$ is smooth and projective over $R$, the line bundle $M$ extends to a line bundle $M_R$ on $X_R$, and the closed point $q\in X_\mathbb{C}(\mathbb{C})$ extends to an $R$-point $q\in X_R(R)$. We also assume the finite group $\Gamma=\on{Pic}(X_R)[n]$ is constant over $R$. We consider the moduli spaces of semistable parabolic Higgs bundles over $X_R$, then the dual pair of Hitchin systems 
\[
(\mathcal{M}^{L_d, \boldsymbol{\alpha}}_{\on{SL}_n}(X_R), \mathcal{M}^{e, \boldsymbol{\alpha}'}_{\on{PGL}_n}(X_R),\mathcal{A}_R, \check{\alpha}^{e'}, \hat{\alpha}^{d'})
\]
is an $R$-model for the dual pair of Hitchin systems over $\mathbb{C}$.

Recall that in Proposition \ref{prop: trializing section exists}, we have constructed a $\Gamma$-invariant gauge form $\omega_d$ on the moduli space $\mathcal{M}^{L_d, \boldsymbol{\alpha}}_{\on{SL}_n}(X)$. By possibly extending $R$, we assume that $\omega_d$ extends to a gauge form $\omega_{R,d}$ on $\mathcal{M}^{L_d, \boldsymbol{\alpha}}_{\on{SL}_n}(X_R)$ and $\omega_e$ extends to a gauge form $\omega_{R, e}$ on $\mathcal{M}^{L_e, \boldsymbol{\alpha}'}_{\on{SL}_n}(X_R)$. The gauge form $\omega_{R, e}$ remains $\Gamma$-invariant and it descends to a gauge form on $\mathcal{M}^{e, \boldsymbol{\alpha}'}_{\on{PGL}_n}(X_R)$. Recall that there is an open subscheme of $\mathcal{M}^{L_d, \boldsymbol{\alpha}}_{\on{SL}_n}(X_R)$ (resp. $\mathcal{M}^{e, \boldsymbol{\alpha}'}_{\on{PGL}_n}(X_R)$) that is a $\check{\mathcal{P}}_R$ (resp. $\hat{\mathcal{P}}_R$)-torsor, and the two abelian schemes $\check{\mathcal{P}}_R$ and $\hat{\mathcal{P}}_R$ are related by the \'etale isogeny
\[
\phi: \check{\mathcal{P}}_R\longrightarrow \check{\mathcal{P}}_R/\Gamma\cong \hat{\mathcal{P}}_R. 
\]
\begin{proof}[Proof of Theorem \ref{thm: main theorem}] We would like to prove the following equality of (stringy) $E$-polynomials:
\begin{equation}
    E(\mathcal{M}^{L_d, \boldsymbol{\alpha}}_{\on{SL}_n}(X); u,v)=E_{st}(\mathcal{M}^{e, \boldsymbol{\alpha}'}_{\on{PGL}_n}(X), \hat{\alpha}^{d'}; u,v). 
\end{equation}
Since $\mathcal{M}^{L_d, \boldsymbol{\alpha}}_{\on{SL}_n}(X)$ is a smooth complex variety, its stringy $E$-polynomial doesn't depend on the choice of the gerbe, therefore it is equivalent to show 
\begin{equation}\label{eq: mirror symmetry}
     E_{st}(\mathcal{M}^{L_d, \boldsymbol{\alpha}}_{\on{SL}_n}(X), \check{\alpha}^{e'}; u,v)=E_{st}(\mathcal{M}^{e, \boldsymbol{\alpha}'}_{\on{PGL}_n}(X), \hat{\alpha}^{d'}; u,v). 
\end{equation}
By Theorem \ref{thm: stringy point count vs E-polynomial}, in order to prove the desired equality (\ref{eq: mirror symmetry}) of (stringy) $E$-polynomials, it is enough to show that for any ring homomorphism $R\to \mathbb{F}_q$, the $\on{SL}_n$ and $\on{PGL}_n$-moduli space have the same twisted point-count. By Theorem \ref{thm: stringy point count=p-adic integral}, equality of stringy point-counts can be reduced to proving the equality of $p$-adic integrals on the $\on{SL}_n$- and $\on{PGL}_n$-moduli space $\mathcal{M}^{L_d, \boldsymbol{\alpha}}_{\on{SL}_n}(X_{\mathcal{O}_F})$ and $\mathcal{M}^{e, \boldsymbol{\alpha}'}_{\on{PGL}_n}(X_{\mathcal{O}_F})$, which are constructed from $\mathcal{M}^{L_d, \boldsymbol{\alpha}}_{\on{SL}_n}(X_R)$ and $\mathcal{M}^{e, \boldsymbol{\alpha}'}_{\on{PGL}_n}(X_R)$ by base-change through $R\to\mathbb{F}_q\to \mathbb{F}_q[[t]]=\mathcal{O}_F$. Note that by Proposition \ref{prop: dual Hitchin system}, 
\[
(\mathcal{M}^{L_d, \boldsymbol{\alpha}}_{\on{SL}_n}(X_R), \mathcal{M}^{e, \boldsymbol{\alpha}'}_{\on{PGL}_n}(X_R),\mathcal{A}_R, \check{\alpha}^{e'}, \hat{\alpha}^{d'})
\]
forms a dual pair of weak abstract Hitchin systems, therefore we can apply discussions in Subsection \ref{subsec: $p$-adic integration and mirror symmetry} to simplify the computation of $p$-adic integrals. 
Recall that we've constructed $\Gamma$-invariant gauge forms $\omega_{R,d}$ on $\mathcal{M}^{L_d, \boldsymbol{\alpha}}_{\on{SL}_n}(X_R)$ and $\omega_{R,e}$ on $\mathcal{M}^{L_e, \boldsymbol{\alpha}'}_{\on{SL}_n}(X_R)$. At the end of Subsection \ref{subsection: abstract hitchin systems}, we've associated with them relative gauge forms $\tilde{\omega}'_{R, d}$ and $\tilde{\omega}'_{R, e}$ on $\check{\mathcal{P}}_R/\mathcal{A}_R$. By Theorem \ref{main thm: p-adic integral for Hitchin} and Remark \ref{rk: p-adic integral for Hitchin}, in order to prove the desired equality of $p$-adic integrals for any base-change $R\to \mathbb{F}_q[[t]]$, it is enough to show that $\tilde{\omega}'_{R, d}=\tilde{\omega}'_{R, e}$. In Proposition \ref{prop: comparision between degrees}, we have proved the corresponding equality $\tilde{\omega}'_{ d}=\tilde{\omega}'_{e}$ of relative gauge forms on $\check{\mathcal{P}}\cong \check{\mathcal{P}}_R\otimes_R\mathbb{C}$. Since $\check{\mathcal{P}}_R$ is smooth over $R$, the equality over $\mathbb{C}$ implies the desired equality over $R$. 
\end{proof} 

\subsection{Proof of independence on degree}
\begin{proof}[Proof of Theorem \ref{thm: main theorem independence} Part (a)] Let $\boldsymbol{\alpha}_1$ (resp. $\boldsymbol{\alpha}_2$) be a set of parabolic weights that is generic for degree $d_1$ (resp. $d_2$). Our goal is to prove the following equality of $E$-polynomials
\begin{equation}\label{eq: E-polynomial not depend on degree}
  E(\mathcal{M}^{{d_1}, \boldsymbol{\alpha}_1}_{\on{GL}_n}(X); u,v)=E(\mathcal{M}^{{d_2}, \boldsymbol{\alpha}_2}_{\on{GL}_n}(X); u,v).    
\end{equation}
For simplicity of notations, we write $\mathcal{M}^{{d_i}, \boldsymbol{\alpha}_i}_{\on{GL}_n}(X)=\mathcal{M}^{d_i}(X)$. Let $X_R$ be an $R$-model for $X$ such that the gauge form $\omega_{d_i}$ constructed in Corollary \ref{cor: gauge form GLn} extends to a gauge form $\omega_{R,d_i}$ on $X_R$ for $i=1,2$. By Theorem \ref{thm: stringy point count vs E-polynomial} and Theorem \ref{thm: stringy point count=p-adic integral}, proving \ref{eq: E-polynomial not depend on degree} is reduced to showing the following equality of $p$-adic integrals
\begin{equation}\label{eq: different degree p-adic}
    \int_{\mathcal{M}^{d_1}(X_{\mathcal{O}_F})(\mathcal{O}_F)}|\omega_{\mathcal{O}_F,d_1}|=\int_{\mathcal{M}^{d_2}(X_{\mathcal{O}_F})(\mathcal{O}_F)}|\omega_{\mathcal{O}_F,d_2}|
\end{equation}
for every ring homomorphism $R\longrightarrow \mathcal{O}_F=\mathbb{F}_q[[t]]$ where $\mathbb{F}_q$ is a finite field. For $i=1,2$, Proposition \ref{prop: properties of p-adic integration} implies
\[
\int_{\mathcal{M}^{d_i}(X_{\mathcal{O}_F})(\mathcal{O}_F)}|\omega_{\mathcal{O}_F,d_i}|= \int_{a\in \mathcal{A}_{\mathcal{O}_F}({\mathcal{O}_F)}^{\flat}}|\omega_{\mathcal{A}_{\mathcal{O}_F}}|\int_{h_{i}^{-1}(a)(F)}|(\tilde{\omega}_{\mathcal{O}_F,d_i})_{a_F}|,
\]
therefore proving (\ref{eq: different degree p-adic}) is reduced to proving
\begin{equation}\label{eq: p-adic integration hichin fiber}
    \int_{h_{1}^{-1}(a)(F)}|(\tilde{\omega}_{\mathcal{O}_F,d_1})_{a_F}|=\int_{h_{2}^{-1}(a)(F)}|(\tilde{\omega}_{\mathcal{O}_F,d_2})_{a_F}|
\end{equation}
for all $a\in \mathcal{A}_{\mathcal{O}_F}(\mathcal{O}_F)\cap \mathcal{A}^0_{\mathcal{O}_F}(F)$ with corresponding $a_F\in \mathcal{A}^0_{\mathcal{O}_F}(F)$. Since $h_{i}^{-1}(a_F)$ is a ${\mathcal{P}}_{a_F}$-torsor, we have
\[ \int_{h_{i}^{-1}(a)(F)}|(\tilde{\omega}_{\mathcal{O}_F,d_i})_{a_F}| = \begin{cases}
  \displaystyle\int_{{\mathcal{P}}_{a_F}(F)}|(\tilde{\omega}'_{\mathcal{O}_F,d_i})_{a_F}| & \text{if } {h_{i}^{-1}(a)(F)} \text{ is non-empty},\\\\
  0, & \text{if } {h_{i}^{-1}(a)(F)} \text{ is empty}.
  \end{cases}
\]
Note that Corollary \ref{cor: gauge form GLn} implies $\tilde{\omega}'_{\mathcal{O}_F,d_1}=\tilde{\omega}'_{\mathcal{O}_F,d_2}$. Therefore proving (\ref{eq: p-adic integration hichin fiber}) is reduced to proving
\[
{h_{1}^{-1}(a_F)(F)} \text{ is non-empty} \Longleftrightarrow {h_{2}^{-1}(a_F)(F)} \text{ is non-empty}. 
\]
By Theorem \ref{theorem: spectral data GLn}, this is the same as showing
\begin{equation}\label{eq: picard non empty}
 \on{Pic}^{d_1-d_{\boldsymbol{m}}}(\Sigma_{a_F})(F) \text{ is non-empty} \Longleftrightarrow \on{Pic}^{d_2-d_{\boldsymbol{m}}}(\Sigma_{a_F})(F) \text{ is non-empty}.    
\end{equation} 
Recall that $\boldsymbol{m}=(m_1, m_2,\dots, m_r)$ is our fixed parabolic multiplicities for the partial flag structure. By Lemma \ref{lemma: criterion for generic weights}, in order to prove (\ref{eq: picard non empty}), it is enough to show that for any integer $d$ such that $\on{gcd}(d, m_1, m_2, \dots, m_r)=1$, we have
\begin{equation}\label{eq: picard non empty 2}
 \on{Pic}^{d-d_{\boldsymbol{m}}}(\Sigma_{a_F})(F) \text{ is non-empty} \Longleftrightarrow \on{Pic}^{1-d_{\boldsymbol{m}}}(\Sigma_{a_F})(F) \text{ is non-empty}.    
\end{equation}

We first fix some notations. Let $\boldsymbol{\lambda}=\lambda_1^{b_1}\lambda_2^{b_2}\cdots\lambda_t^{b_t}$ be the partition of $n$ conjugate to $\boldsymbol{m}$, where $\lambda_1, \lambda_2,\dots, \lambda_t$ are distinct positive integers and $b_i$ is the multiplicity of $\lambda_i$. We denote $m_{\on{gcd}}=\on{gcd}(m_1,m_2,\dots, m_r)$ the greatest common divisor of $m_1, m_2, \dots, m_r$. Note that since $\boldsymbol{m}$ and $\boldsymbol{\lambda}$ are conjugate partitions, we also have $m_{\on{gcd}}=\on{gcd}(b_1, b_2,\dots, b_t)$. 

The proof of (\ref{eq: picard non empty 2}) relies on the geometry of $\Sigma_{a_F}$ described in Remark \ref{rk: singular spectral curves}. Above the marked point $q$, the spectral curve $\widetilde{X}_{a_F}$ at $a_F\in\mathcal{A}^0(F)$ locally looks like $\on{Spec}F[[x,y]]/(f)$ for some $f\in F[[x,y]]$. By Remark \ref{rk: singular spectral curves}, $f$ can be factorized as 
\[ f= \prod_{i=1}^{t}\prod_{j=1}^{b_i}(y^{\lambda_i}-a_{ij}x), a_{ij}\in \bar{F}[[x,y]]^{\times}
\]
in $\bar{F}[[x,y]]$. We denote $f_i=\displaystyle\prod_{j=1}^{b_i}(y^{\lambda_i}-a_{ij}x)$ for $i=1,2,\dots, t$. Since the coefficients of $f$ lie $F$, we have $f_i\in F[[x,y]]$ for each $i$. Recall that $\Sigma_{a_F}$ is the normalization of the spectral curve $\widetilde{X}_{a_F}$ which can be obtained by consecutive blow-ups. In this procedure, the local branches corresponding to $f_1, f_2,\dots,f_t$ are separated apart. It follows that $\Sigma_{a_F}$ admits a Cartier divisor of degree $b_i$ for each $i=1,2,\dots t$. Since $m_{\on{gcd}}=\on{gcd}(b_1, b_2,\dots, b_t)$, we have $\on{Pic}^{m_{\on{gcd}}}(\Sigma_{a_F})(F)$ is non-empty. 

Now we prove (\ref{eq: picard non empty 2}). Let $d$ be an integer such that $\on{gcd}(d, m_1, m_2, \dots, m_r)=1$. Let $a,b$ be integers such that $ad+bm_{\on{gcd}}=1$. Recall that \[d_{\boldsymbol{m}}=-\displaystyle\frac{1}{2}n(n-1)(\on{deg}M+1)+\displaystyle\sum_{i=1}^{r}\frac{1}{2}m_i(m_i-1).
\] 
We note that $2d_{\boldsymbol{m}}$ is divisible by $m_{\on{gcd}}$. We further observe that $(a-1)d_{\boldsymbol{m}}$ is divisible by $m_{\on{gcd}}$. Indeed, if $m_{\on{gcd}}$ is odd, the statement follows from $2d_{\boldsymbol{m}}$ being divisible by $m_{\on{gcd}}$; if $m_{\on{gcd}}$ is even, then $(a-1)$ must also be even. Similarly,  $(d-1)d_{\boldsymbol{m}}$ is also divisible by $m_{\on{gcd}}$. For the $``\Longrightarrow"$ direction of (\ref{eq: picard non empty 2}), let $L_1\in \on{Pic}^{d-d_{\boldsymbol{m}}}(\Sigma_{a_F})(F)$ and $L_2\in \on{Pic}^{m_{\on{gcd}}}(\Sigma_{a_F})(F)$, then $(L_1)^a(L_2)^{b+\frac{(a-1)d_{\boldsymbol{m}}}{m_{\on{gcd}}}}$ gives an element in $\on{Pic}^{1-d_{\boldsymbol{m}}}(\Sigma_{a_F})(F)$. For the $``\Longleftarrow"$ direction of (\ref{eq: picard non empty 2}), let $L_1\in \on{Pic}^{1-d_{\boldsymbol{m}}}(\Sigma_{a_F})(F)$ and $L_2\in \on{Pic}^{m_{\on{gcd}}}(\Sigma_{a_F})(F)$, then $(L_1)^d(L_2)^{\frac{(d-1)d_{\boldsymbol{m}}}{m_{\on{gcd}}}}$ gives an element in $\on{Pic}^{d-d_{\boldsymbol{m}}}(\Sigma_{a_F})(F)$.
\end{proof}

\begin{proof}[Proof of Theorem \ref{thm: main theorem independence} Part (b)]
In Corollary \ref{cor: gauge form GLn}, we've constructed gauge form $\omega_{d_i}$ on $\mathcal{M}^{d_i}(X)$ for $i=1,2$. Let $\mathcal{O}_F=\mathbb{F}_p[[t]]$, let $X_{\mathcal{O}_F}=X\times_{\on{Spec}\mathbb{F}_p}\on{Spec}\mathcal{O}_F$ and let $\omega_{\mathcal{O}_F, {d_i}}$ be the pull-back of $\omega_{d_i}$ to $\mathcal{M}^{d_i}(X_{\mathcal{O}_F})$. By Theorem \ref{thm: stringy point count=p-adic integral}, proving $\#\mathcal{M}^{d_1}(X)=\#\mathcal{M}^{d_2}(X)$ can be reduced to proving the same equality of $p$-adic integral as in (\ref{eq: different degree p-adic}). Same arguments as in the proof of Theorem \ref{thm: main theorem independence} Part (a) lead to the desired result. 
\end{proof}

\end{document}